\documentclass[a4paper]{amsart}
\usepackage{amsmath,amsthm,amssymb,latexsym,epic,bbm,comment}
\usepackage{graphicx,enumerate,stmaryrd}
\usepackage[all,2cell]{xy}
\xyoption{2cell}

\usepackage[active]{srcltx}
\usepackage[parfill]{parskip}
\UseTwocells

\newtheorem{theorem}{Theorem}
\newtheorem{lemma}[theorem]{Lemma}

\newtheorem{corollary}[theorem]{Corollary}
\newtheorem{proposition}[theorem]{Proposition}
\newtheorem{example}[theorem]{Example}

\newtheorem{conjecture}[theorem]{Conjecture}

\newcommand{\tto}{\twoheadrightarrow}

\font\sc=rsfs10
\newcommand{\cC}{\sc\mbox{C}\hspace{1.0pt}}

\newcommand{\cI}{\sc\mbox{I}\hspace{1.0pt}}

\newcommand{\cS}{\sc\mbox{S}\hspace{1.0pt}}

\newcommand{\cD}{\sc\mbox{D}\hspace{1.0pt}}

\newcommand{\cA}{\sc\mbox{A}\hspace{1.0pt}}

\font\scc=rsfs7
\newcommand{\ccC}{\scc\mbox{C}\hspace{1.0pt}}

\begin{document}

\title[Diagrams and discrete extensions ]{Diagrams and discrete extensions\\ for finitary $2$-representations}
\author{Aaron Chan and Volodymyr Mazorchuk}

\begin{abstract}
In this paper we introduce and investigate the notions of diagrams and
discrete extensions in the study of finitary $2$-representations of 
finitary $2$-categories.
\end{abstract}

\maketitle

\section{Introduction and description of the results}\label{s1}

Abstract $2$-representation theory emerged from the ideas of, among others, the papers \cite{BFK,CR,Ro,Ro2},
following growing success of application of the categorification philosophy originating in  \cite{Cr,CF}
to various areas of mathematics, see e.g. \cite{Kh,St,CR}. A systematic study of finitary $2$-representations
of finitary $2$-categories, which is a natural $2$-analogue of the study of finite dimensional modules
over finite dimensional algebras, was initiated in \cite{MM1} and then continued in 
\cite{MM2,MM3,MM4,MM5,MM6,Xa,Zh1,Zh2,GM1,GM2,Zi}, see also references therein. For applications
of this theory,  see, in particular,  \cite{KiMa}.

The main emphasis in the above mentioned papers was put on finding a ``correct'' $2$-analogue for the notion
of a simple module. The most reasonable candidate, called {\em simple transitive $2$-representation} was
proposed in \cite{MM5} where it was also shown that such $2$-representations can be classified for
a large class of finitary  $2$-categories (further classification results of this kind were obtained in 
\cite{Zh2,Zi}). In \cite{MM6}, the classification from \cite{MM5} was extended to a larger class of both,
finitary $2$-representation and finitary $2$-categories, in the spirit of some results from \cite{Ro2}.
The structure of other classes of $2$-representations is, at the moment, completely unclear and very
few methods for construction and study of such representations exist.

The main motivation of the present paper is to develop some basic combinatorial tools for construction 
and study of arbitrary finitary $2$-representations of finitary $2$-categories. We primarily focus on 
finding a proper $2$-analogue for the classical representation theoretic notions of 
\begin{itemize}
\item the diagram of a module (e.g. in the sense of \cite{Al});
\item the first extension group between two modules.
\end{itemize}
These two problems are closely related as edges in the diagram of a module usually correspond to certain
first extension groups between simple subquotients.

The main difference between $2$-representation theory and classical representation theory is that
$2$-representations have two layers of different flavor: a discrete layer of objects and a linear layer
of morphisms. The outcome of this is that our $2$-analogue of the diagram for a module is fairly 
easy to define and describe. The reason for that is the fact that this diagram is related to the
discrete layer which makes things easy and allows the diagram to carry all necessary information 
(in a strong contrast to the general case of \cite{Al}). The situation with the first extension groups 
is exactly the opposite. The discrete environment in which we try to formulate our definition does
not allow us to use any classical tool (the latter strongly depend on the group version of additivity as opposed
to the semigroup version which we have). Put shortly: the nature of the difficulties lies in the fact that
one cannot subtract functors.

We propose the notion which we call {\em discrete extension} and which is designed to be a certain 
combinatorial tool allowing to record some information about how different $2$-representations can 
be ``glued together''. It is by far not as good as its classical linear analogue. Nevertheless, 
we show that this notion has some interesting and non-trivial properties and can be applied 
to understand and construct  new classes of $2$-representations.

The paper is organized as follows: In Section~\ref{s2} we collected all necessary preliminaries
on finitary $2$-categories and their $2$-representations and provide all relevant references. 
In Section~\ref{s2-5} we introduce and study the notion of the apex for a
transitive $2$-representation of a finitary  $2$-category. This is inspired by the corresponding
notion in representation theory of semigroups, see e.g. \cite{GMS} and references therein.
In Section~\ref{s3} we define and study the notion of a diagram for a finitary $2$-representation 
of a finitary  $2$-category. The main result in this section is Theorem~\ref{thm6} which provides 
a non-trivial sufficient condition for a diagram of a $2$-representation to be a disjoint union 
of vertices. We illustrate this theorem by applying it to two examples related to $2$-categories 
of Soergel bimodules over (non simply laced) Dihedral groups.

Section~\ref{s4} introduces the notion of the {\em decorated} diagram for a 
finitary $2$-representation of a finitary  $2$-category and uses this notion to define discrete
extensions of $2$-representation. In Section~\ref{s5} we study, in detail, decorated diagrams and
discrete extensions for $2$-representations of the $2$-category $\cC_A$ of projective functors for a
finite dimensional algebra $A$. Here, Theorem~\ref{thm32} provides a rough description of
discrete extensions between  cell $2$-representations of $\cC_A$ and Subsection~\ref{s5.3}
describes some possibilities for decorated diagrams which can appear in the study of arbitrary 
finitary $2$-representations of  $\cC_A$.

The last section of the paper, Section~\ref{s6}, contains one general result and two interesting
examples. The general result,  Theorem~\ref{thm41}, gives a non-trivial sufficient condition for the discrete
extension set between two simple transitive $2$-representation to be empty. Subsection~\ref{s6.2}
provides an explicit construction of an unexpected non-trivial discrete extension between two 
cell $2$-representations for the $2$-category of Soergel bimodules of type $A_2$. The unexpected 
feature of this example is due to the fact that the left cells corresponding to these two 
$2$-representations are not neighbors with respect to the left Kazhdan-Lusztig order on the
corresponding Coxeter group. The last subsection provides a similar construction for 
the $2$-category obtained as the external tensor product $\cC_A\boxtimes \cC_{A^{\mathrm{op}}}$.
\vspace{5mm}

{\bf Acknowledgment.} The second author is partially supported by the Swedish Research Council.
We thank Vanessa Miemietz  and Jakob Zimmermann for very helpful and stimulating discussions.

\section{Finitary $2$-categories and their $2$-representations}\label{s2}

\subsection{Basic notation}\label{s2.1}

We work over a fixed algebraically closed field $\Bbbk$ and we try to follow all notational conventions
from \cite{MM6}. We abbreviate $\otimes_{\Bbbk}$ by $\otimes$.
We also refer the reader to \cite{Le,McL} for generalities on $2$-categories.

\subsection{Finitary categories}\label{s2.2}

A $\Bbbk$-linear category is called {\em finitary} if it is additive, idempotent split and Krull-Schmidt
with finitely many isomorphism classes of indecomposable objects and finite dimensional $\Bbbk$-vector 
spaces of morphisms.

A {\em finitary} $2$-category $\cC$ is a $2$-category with finitely many objects whose morphism categories
are finitary $\Bbbk$-linear categories (with composition preserving both the additive and the $\Bbbk$-linear
structures), together with the additional assumption that all identity 
$1$-morphisms in $\cC$ are indecomposable.

A {\em finitary} $2$-representation of a finitary $2$-category $\cC$ is a (strict) $2$-functor 
$\mathbf{M}$ which assigns to each object of $\cC$ a (small) finitary $\Bbbk$-linear category, 
to each $1$-morphism an additive functor, and to each $2$-morphism a natural transformation
of functors. We denote by $\cC\text{-}\mathrm{afmod}$ the $2$-category of all 
finitary $2$-representations of $\cC$.  In the category $\cC\text{-}\mathrm{afmod}$, $1$-morphisms
are $2$-natural transformations and $2$-morphisms are modifications. Two $2$-representations
are called {\em equivalent} if there is a $2$-natural transformation between them whose value at
any object is an equivalence of categories.

For every $\mathtt{i}\in\cC$, we denote by $\mathbf{P}_{\mathtt{i}}$ the corresponding {\em principal}
$2$-representation $\cC(\mathtt{i},{}_-)$ of $\cC$.

Given a finitary $2$-representation $\mathbf{M}$ of $\cC$ and any small finitary $\Bbbk$-linear category
$\mathcal{A}$, we denote by  $\mathbf{M}^{\boxtimes\mathcal{A}}$ the {\em inflation} of $\mathbf{M}$
by $\mathcal{A}$ as described in \cite[Subsection~3.6]{MM6}.

\subsection{Combinatorics}\label{s2.3}

We denote by $\mathcal{S}[\cC]$ the set of isomorphism classes of indecomposable $1$-morphisms in $\cC$. 
This has the natural structure of a {\em multisemigroup}, see \cite{MM2} for details and \cite{KM} 
for generalities on multisemigroups. We denote by $\leq_L$, $\leq_R$ or $\leq_J$ the
{\em left}, {\em right} and {\em two-sided} preorders on $\mathcal{S}[\cC]$, respectively. 
Equivalence classes for these preorders are called respective {\em cells} and the corresponding 
equivalence relations are denoted $\sim_L$, $\sim_R$ and $\sim_J$, respectively.

A two-sided cell $\mathcal{J}$ is called {\em strongly regular} provided that 
\begin{itemize}
\item all left (resp. right) cells inside  $\mathcal{J}$ are not comparable with respect 
to the left (resp. right) order;
\item the intersection of any left and any right cell inside $\mathcal{J}$ consists of 
exactly one isomorphism class of $1$-morphisms.
\end{itemize}
A two-sided cell $\mathcal{J}$ of a finitary $2$-category $\cC$ is called {\em idempotent} provided that
there are $1$-morphisms $\mathrm{F}$, $\mathrm{G}$ and $\mathrm{H}$ in $\mathcal{J}$ such that 
$\mathrm{H}$ is isomorphic to a direct summand of $\mathrm{F}\circ \mathrm{G}$. This notion is 
an analogue of the notion of a regular $\mathcal{J}$-class for finite semigroups. 

For a left cell $\mathcal{L}$, we denote by $\mathtt{i}_{\mathcal{L}}$ the object in $\cC$ with the
property that all $1$-morphisms in $\mathcal{L}$ start at this object.
For a left (right, two-sided) cell $\mathcal{X}$, we denote by $\mathrm{F}{(\mathcal{X})}$ a multiplicity-free 
direct sum of all indecomposable $1$-morphisms in $\mathcal{X}$.
For a left cell $\mathcal{L}$, we denote by $\mathbf{C}_{\mathcal{L}}$ the corresponding 
{\em cell $2$-representation} as in \cite[Subsection~6.5]{MM2}, see also \cite[Subsection~4.5]{MM1}.

\subsection{Abelian $2$-representations and abelianization}\label{s2.4}

An {\em abelian} $2$-representation of a finitary $2$-category $\cC$ is a (strict) $2$-functor 
$\mathbf{M}$ which assigns to each object of $\cC$ a (small) $\Bbbk$-linear category equivalent
to the module category of a finite dimensional associative and unital $\Bbbk$-algebra, 
to each $1$-morphism a right exact functor, and to each $2$-morphism a natural transformation
of functors. The $2$-category of abelian $2$-representations of $\cC$ is denoted by
$\cC\text{-}\mathrm{mod}$. 

There is a diagrammatic {\em abelianization} $2$-functor  from $\cC\text{-}\mathrm{afmod}$ to 
$\cC\text{-}\mathrm{mod}$ as described in \cite[Subsection~2.7]{MM6}. It sends
a finitary $2$-representation $\mathbf{M}$ to an abelian $2$-representation
$\overline{\mathbf{M}}$.

\subsection{Weakly fiat $2$-categories}\label{s2.5}

A finitary $2$-category $\cC$ is called {\em weakly fiat} provided that it has a weak 
antiautoequivalence $*$ (which reverses both $1$-morphisms and $2$-morphisms) such that,
for any $1$-morphism $\mathrm{F}$, there are $2$-morphisms between $\mathrm{F}\mathrm{F}^*$
(resp. $\mathrm{F}^*\mathrm{F}$) and the corresponding identity $1$-morphisms which make
the pair $(\mathrm{F},\mathrm{F}^*)$ into a pair of adjoint $1$-morphisms
(cf. the notion of a {\em rigid} tensor category in \cite{EGNO}).

A {\em fiat} $2$-category is a weakly fiat $2$-category in which the weak
antiautoequivalence is a weak involution.

\subsection{Decategorification}\label{s2.6}

Let $\cC$ be a finitary $2$-category. The {\em (Grothendieck) decategorification} of $\cC$ 
is the category $[\cC]_{\oplus}$ with the same objects as $\cC$ and such that 
$[\cC]_{\oplus}(\mathtt{i},\mathtt{j}):=[\cC(\mathtt{i},\mathtt{j})]_{\oplus}$, the
split Grothendieck group of $\cC(\mathtt{i},\mathtt{j})$, with induced composition.

For a two-sided cell $\mathcal{J}$, we denote by $B^{(\mathcal{J})}$ the $\mathbb{R}$-algebra
with a basis given by isomorphism classes of indecomposable $1$-morphisms in $\mathcal{J}$
and the product induced from the composition of $1$-morphisms modulo all larger
two-sided cells. Note that $B^{(\mathcal{J})}$
is not a unital algebra in general. The algebra $B^{(\mathcal{J})}$ is {\em positively based}
in the sense that all structure constants with respect to the defining basis 
(of isomorphism classes of indecomposable $1$-morphisms in $\mathcal{J}$) are
non-negative real numbers. We denote by $B^{(\mathcal{J})}_+$ the {\em positive cone}
in $B^{(\mathcal{J})}$, that is the set of all elements which are linear combinations of basis
elements with positive real coefficients.

If $\mathbf{M}$ is a  finitary $2$-representation of $\cC$, then the 
{\em (Grothendieck) decategorification} of $\mathbf{M}$ is a representation of $[\cC]_{\oplus}$
given by the induced  action of $[\cC]_{\oplus}$ on free abelian groups $[\mathbf{M}(\mathtt{i})]_{\oplus}$,
where $\mathtt{i}\in\cC$. Note that a {\em standard} basis in $[\mathbf{M}(\mathtt{i})]_{\oplus}$
is given by the isomorphism classes of indecomposable objects in $\mathbf{M}(\mathtt{i})$.

For each $1$-morphism $\mathrm{F}\in \cC$, we will denote by $[\mathrm{F}]_{\oplus}=
[\mathrm{F}]_{\oplus}^{\mathbf{M}}$ the matrix
of the linear operator $\mathrm{F}$ in a standard basis mentioned above (note that this matrix depends
on the ordering of elements in the standard basis).

If $\mathbf{M}$ is an abelian $2$-representation of $\cC$, then the 
{\em (Grothendieck) decategorification} of $\mathbf{M}$ is a representation of $[\cC]_{\oplus}$
given by the induced  action of $[\cC]_{\oplus}$ on Grothendieck groups $[\mathbf{M}(\mathtt{i})]$
of the corresponding abelian categories, where $\mathtt{i}\in\cC$. Note that a {\em standard} 
basis in $[\mathbf{M}(\mathtt{i})]$ is given by the isomorphism classes of simple objects in 
$\mathbf{M}(\mathtt{i})$. This is well-defined provided that the functor $\mathbf{M}(\mathrm{F})$ is exact,
for any $1$-morphism $\mathrm{F}\in\cC$, in particular, this is always well-defined if $\cC$ is weakly fiat.

For each $1$-morphism $\mathrm{F}\in \cC$, we will denote by $\llbracket\mathrm{F}\rrbracket$ the matrix
of the linear operator $\mathrm{F}$ in a standard basis mentioned above (note that this matrix depends
on the ordering of elements in the standard basis).

\section{Transitive $2$-representations and the apex}\label{s2-5}

\subsection{The action preorder}\label{s3.1}

Let $\cC$ be a finitary $2$-category and $\mathbf{M}$ a finitary $2$-representation of $\cC$.
Denote by $\mathrm{Ind}(\mathbf{M})$ the set of isomorphism classes of indecomposable objects
in all $\mathbf{M}(\mathtt{i})$, where $\mathtt{i}\in\cC$. For simplicity, we identify
objects with their corresponding isomorphism classes. For $X,Y\in \mathbf{M}(\mathtt{i})$,
we write $X\boldsymbol{\to} Y$ provided that there is a $1$-morphism  $\mathrm{F}\in \cC$ such that 
$Y$ is isomorphic to a direct summand of $\mathrm{F}\, X$. The relation $\boldsymbol{\to}$ is,
clearly, a preorder on $\mathrm{Ind}(\mathbf{M})$, called the {\em action preorder}.
The $2$-representation $\mathbf{M}$ is said to be {\em transitive} provided that 
$\boldsymbol{\to}$ is the full relation. A transitive $2$-representation of $\cC$ is called 
{\em simple} provided that it does not have any proper ideal which is invariant under the
$2$-action of $\cC$.

For $X,Y\in \mathbf{M}(\mathtt{i})$, we write $X\boldsymbol{\leftrightarrow}Y$ provided that 
$X\boldsymbol{\to} Y$ and $Y\boldsymbol{\to} X$. Then $\boldsymbol{\leftrightarrow}$ is an equivalence
relation and the preorder $\boldsymbol{\to}$ induces a partial order on the set 
$\mathrm{Ind}(\mathbf{M})/_{\boldsymbol{\leftrightarrow}}$ of equivalence classes with respect to
$\boldsymbol{\leftrightarrow}$. The elements in $\mathrm{Ind}(\mathbf{M})/_{\boldsymbol{\leftrightarrow}}$
correspond to subquotients in a weak Jordan-H{\"o}lder series of $\mathbf{M}$ as defined in 
\cite[Section~4]{MM5}.

\subsection{Apex}\label{s2-5.1}

Let $\cC$ be a finitary $2$-category and
$\mathbf{M}$ a finitary $2$-representation of $\cC$. Recall that the {\em annihilator} 
$\mathrm{Ann}_{\ccC}(\mathbf{M})$ is the two-sided ideal in $\cC$ which consists of all 
$2$-morphism annihilated by $\mathbf{M}$, see \cite[Subsection~4.2]{MM1}. We can now consider the
finitary $2$-category $\cC_{\mathbf{M}}:=\cC/\mathrm{Ann}_{\ccC}(\mathbf{M})$.

\begin{lemma}\label{lem7}
Let $\cC$ be a finitary $2$-category and $\mathbf{M}$ a transitive $2$-representation of $\cC$.
Then $\mathcal{S}[\cC_{\mathbf{M}}]$ has a unique maximal two-sided cell. This two-sided cell
is idempotent.
\end{lemma}

Note that two-sided cells of $\mathcal{S}[\cC_{\mathbf{M}}]$ are also two-sided cells of 
$\mathcal{S}[\cC]$.  At the same time, some two-sided cells of $\cC$ can disappear in $\cC_{\mathbf{M}}$. 
All the latter two-sided cells form an upper set in the poset $(\mathcal{S}[\cC],\leq_J)$.
The unique two-sided cell which corresponds to $\mathbf{M}$ by Lemma~\ref{lem7} is called the
{\em apex} of  $\mathbf{M}$, cf. \cite[Definition~4]{GMS} for the corresponding notion in
representation theory of semigroups. For origins of this notion, see \cite{Mu}.

\begin{proof}
For simplicity, we assume that $\cC$ has only one object $\mathtt{i}$, the general case is
proved similarly.

Assume that $\mathcal{J}_1$ and $\mathcal{J}_2$ are two different maximal two-sided cells 
in $\mathcal{S}[\cC_{\mathbf{M}}]$.  Let $\mathrm{F}$ (resp. $\mathrm{G}$) be a multiplicity-free 
direct sum of representatives of  isomorphism classes of indecomposable $1$-morphisms in 
$\mathcal{J}_1$ (resp. $\mathcal{J}_2$). Then, by maximality of $\mathcal{J}_1$ and $\mathcal{J}_2$, 
we have $\mathrm{F}\circ\mathrm{G}=0$ and $\mathrm{G}\circ\mathrm{F}=0$. Let $X$ be a multiplicity-free
direct sum of representatives of isomorphism classes of indecomposable objects in all
$\mathbf{M}(\mathtt{i})$. Then $X$ is an additive generator of $\mathbf{M}(\mathtt{i})$.

Note that $\mathrm{F}\, X\neq 0$ by assumptions. As $\mathcal{J}_1$ is maximal and $\mathbf{M}$ is transitive,
we get $\mathrm{add}(\mathrm{F}\, X)=\mathrm{add}(X)$. Similarly $\mathrm{add}(\mathrm{G}\, X)=\mathrm{add}(X)$.
Therefore we have the equality $\mathrm{add}(\mathrm{F}\circ \mathrm{G}\, X)=\mathrm{add}(X)$, which contradicts
$\mathrm{F}\circ\mathrm{G}=0$ established in the previous paragraph. The first claim of the lemma follows.

From $\mathrm{add}(\mathrm{F}\, X)=\mathrm{add}(X)$, we have 
$\mathrm{add}((\mathrm{F}\circ \mathrm{F})\, X)=\mathrm{add}(\mathrm{F}\, X)=\mathrm{add}(X)$
and hence $\mathrm{F}\circ \mathrm{F}\neq 0$, which implies that 
the above unique maximal two-sided cell is idempotent.
\end{proof}

\begin{example}\label{ex7-1}
{\rm  
Let $\cC$ be a finitary $2$-category and $\mathcal{J}$ an idempotent two-sided cell for $\cC$.
Let $\mathcal{L}$ in $\mathcal{J}$ be a left cell which is maximal (inside $\mathcal{J}$) with
respect to $\leq_L$. Let $\mathbf{C}_{\mathcal{L}}$ be the corresponding
cell $2$-representation as in \cite[Subsection~6.5]{MM2}. This $2$-representation is transitive
and we claim that $\mathcal{J}$ is its apex. To see this, we need to show that 
$\mathcal{J}$ does not annihilate $\mathbf{C}_{\mathcal{L}}$. Without loss of generality we may 
assume that $\mathcal{J}$ is the unique maximal two-sided cell. 
Then, by maximality of $\mathcal{L}$ with respect to $\leq_L$, we just need to show that 
$\mathcal{J}\circ \mathcal{L}\neq 0$. If
$\mathcal{J}\circ \mathcal{L}=0$, then $\mathcal{J}\circ \mathcal{L}\circ \mathcal{S}[\cC]=0$. However,
$\mathcal{L}\circ \mathcal{S}[\cC]$ is a non-zero two-sided ideal containing $\mathcal{L}$ and therefore
also containing $\mathcal{J}$ by the maximality of the latter with respect to $\leq_J$. 
As $\mathcal{J}$ is idempotent, we have $\mathcal{J}\circ \mathcal{J}\neq 0$, which gives a contradiction.
}
\end{example}

The following proposition relates the apex of a transitive $2$-representation $\mathbf{M}$  to the apex of 
the simple transitive quotient of $\mathbf{M}$.

\begin{proposition}\label{prop23}
Let $\mathbf{M}$ be a transitive $2$-representations of $\cC$ and $\mathbf{N}$ be the simple 
transitive quotient of $\mathbf{M}$. Then $\mathbf{M}$ and $\mathbf{N}$ have the same apex.
\end{proposition}

\begin{proof}
If $\mathbf{N}(\mathrm{F})\neq 0$, then, clearly, $\mathbf{M}(\mathrm{F})\neq 0$.
Conversely, assume that $\mathbf{M}(\mathrm{F})\neq 0$.
Then there is an indecomposable object $X$ in some $\mathbf{M}(\mathtt{i})$ such that 
$\mathbf{M}(\mathrm{F})\, X\neq 0$. Let $Y$ be an indecomposable direct summand of $\mathbf{M}(\mathrm{F})\, X$.
Then, by construction, the image of $Y$ in $\mathbf{N}(\mathtt{i})$ is non-zero and hence is  an
indecomposable direct summand of $\mathbf{N}(\mathrm{F})\, X$. Therefore $\mathbf{N}(\mathrm{F})\neq 0$
and the claim follows.
\end{proof}

\subsection{Properties of the apex}\label{s2-5.2}

\begin{corollary}\label{cor8}
Let $\cC$ be a finitary $2$-category and $\mathbf{M}$ a transitive $2$-representation of $\cC$
with apex $\mathcal{J}$. Then, for any left cell $\mathcal{L}$ in $\mathcal{J}$ 
maximal (inside $\mathcal{J}$) with respect to $\leq_L$  and any $\mathtt{j}\in\cC$,
applying 
\begin{displaymath}
\bigoplus_{\mathrm{F}\in\mathcal{L}\cap \ccC(\mathtt{i}_{\mathcal{L}},\mathtt{j})}\mathbf{M}(\mathrm{F})
\end{displaymath}
to an additive generator of $\mathbf{M}(\mathtt{i}_{\mathcal{L}})$ 
gives an additive generator in  $\mathbf{M}(\mathtt{j})$.
\end{corollary}

\begin{proof}
Because of maximality of both $\mathcal{J}$ and $\mathcal{L}$, the additive closure of all
$1$-morphisms in $\mathcal{L}$ gives a left ideal in $\cC$ and hence the additive
closure of all objects obtained by applying all $1$-morphisms in $\mathcal{L}$ gives
a $2$-subrepresentation of $\mathbf{M}$. Note that $\mathbf{M}(\mathcal{L})$
is non-zero by our assumptions. Since $\mathbf{M}$ is transitive, the
$2$-subrepresentation must therefore coincide with $\mathbf{M}$. The claim follows.
\end{proof}

\begin{lemma}\label{lem801}
Let $\cC$ be a finitary $2$-category, $\mathbf{M}$ a transitive $2$-representation of $\cC$
and $\mathcal{J}$ the apex of $\mathbf{M}$. Let, further,  $\mathrm{F}$ be a direct sum of
indecomposable $1$-morphisms in $\mathcal{J}$ in which each indecomposable $1$-morphisms from $\mathcal{J}$
occurs with positive multiplicity. Then all coefficients in the matrix $[\mathrm{F}]_{\oplus}$
are positive.
\end{lemma}

\begin{proof}
Let $X$ be an indecomposable object in some $\mathbf{M}(\mathtt{i})$. Then the additive closure of 
$\mathrm{F}\, X$ is, clearly, invariant under the action of $\cC$. Because of the transitivity of 
$\mathbf{M}$, this additive closure either covers all $\mathbf{M}(\mathtt{j})$, where  
$\mathtt{j}\in\cC$, or this additive closure is zero. In the first case the column in 
$[\mathrm{F}]_{\oplus}$ corresponding to $X$ obviously has only positive entries. We claim the the second 
case does not occur.

To prove that the second case does not occur, without loss of generality we may assume that 
$\mathcal{J}$ is the unique maximal two-sided cell for $\cC$. Suppose that $\mathrm{F}\, X=0$.
For any $\mathrm{G}\in \mathcal{J}$ and $\mathrm{H}$ in $\cC$, each indecomposable direct 
summand of $\mathrm{G}\circ \mathrm{H}$ is in $\mathcal{J}$. Therefore $\mathrm{G}\circ \mathrm{H}\,X=0$.
On the other hand, the additive closure of $\mathrm{H}\,X$, where $\mathrm{H}$ in $\cC$, 
covers all $\mathbf{M}(\mathtt{j})$ by transitivity of $\mathbf{M}$. As 
$\mathbf{M}(\mathrm{F})\neq 0$, we thus get that the additive closure of
$\mathrm{F}\circ \mathrm{H}\,X$, where $\mathrm{H}$ in $\cC$, cannot be zero.
The obtained contradiction completes the proof.
\end{proof}

\subsection{Simple transitive $2$-representations and $\mathcal{J}$-simple quotients}\label{s2-5.3}

Let $\cC$ be a finitary $2$-category and $\mathcal{J}$ be a two-sided cell in $\cC$. Recall from 
\cite[Subsection~6.2]{MM2} that $\cC$ is called {\em $\mathcal{J}$-simple} provided that any non-zero
two-sided $2$-ideal in $\cC$ contains the identity $2$-morphisms for all $1$-morphisms in $\mathcal{J}$.
In other terms, there is a maximum element, denoted  $\cI_{\mathcal{J}}$, in the set of all ideals in $\cC$
which do not contain the identity $2$-morphisms for $1$-morphisms in $\mathcal{J}$; the $2$-category 
$\cC$ is $\mathcal{J}$-simple if and only if $\cI_{\mathcal{J}}=0$.

A $2$-representation $\mathbf{M}$ is said to be {\em $\mathcal{J}$-radical} provided that, for any 
$1$-morphisms $\mathrm{F}$ and $\mathrm{G}$ in $\mathcal{J}$ and any $2$-morphism
$\alpha:\mathrm{F}\to \mathrm{G}$ such that
$\alpha\in\cI_{\mathcal{J}}$, the evaluation of the natural transformation  
$\mathbf{M}(\alpha)$ at any object gives a radical morphism.

\begin{proposition}\label{prop9}
Let $\cC$ be a finitary $2$-category and  $\mathbf{M}$ a simple transitive $2$-representation of $\cC$
with apex $\mathcal{J}$. Then $\mathcal{J}$ is, naturally, a two-sided cell of 
$\cC/\mathrm{Ann}_{\ccC}(\mathbf{M})$. Moreover, if $\mathbf{M}$ is $\mathcal{J}$-radical,
then the $2$-category $\cC/\mathrm{Ann}_{\ccC}(\mathbf{M})$ is $\mathcal{J}$-simple.
\end{proposition}

\begin{proof}
Set $\cA:=\cC/\mathrm{Ann}_{\ccC}(\mathbf{M})$. That $\mathcal{J}$ is a two-sided cell of  $\cA$, 
follows from definitions. Let $\mathcal{I}$ be the  two-sided ideal  in
\begin{displaymath}
\mathcal{C}:=\coprod_{\mathtt{i}\in\ccC}\mathbf{M}(\mathtt{i}), 
\end{displaymath}
generated by $\mathbf{M}(\alpha)_X$, for $\alpha\in \cI_{\mathcal{J}}$ and $X\in\mathrm{Ind}(\mathbf{M})$.
Then $\mathcal{I}$ is stable under the   action of $\cC$, by construction.
We claim that the ideal $\mathcal{I}$ is a proper ideal in $\mathcal{C}$. Indeed,
as $\mathbf{M}$ is $\mathcal{J}$-radical, the ideal $\mathcal{I}$ belongs to the radical of $\mathcal{C}$
and hence cannot coincide with $\mathcal{C}$. Now, from simple transitivity of $\mathbf{M}$, 
we get $\mathcal{I}=0$. This means that $\mathbf{M}(\alpha)=0$, for all $\alpha\in \cI_{\mathcal{J}}$.
Therefore $\cI_{\mathcal{J}}\subset \mathrm{Ann}_{\ccC}(\mathbf{M})$. In particular, the 
image of $\cI_{\mathcal{J}}$ in $\cA$, which is also the maximum ideal of $\cA$ not containing 
$\mathrm{id}_{\mathrm{F}}$, for any $\mathrm{F}\in\mathcal{J}$, is zero. Therefore
$\cA$ is $\mathcal{J}$-simple.
\end{proof}
\vspace{7mm}

\section{Diagrams for finitary $2$-representations}\label{s3}

\subsection{Diagrams}\label{s3.2}

Let $\cC$ be a finitary $2$-category and $\mathbf{M}$ a finitary $2$-representation of $\cC$.
Then the {\em diagram} of $\mathbf{M}$ is defined to be the Hasse diagram of the poset
$(\mathrm{Ind}(\mathbf{M})/_{\boldsymbol{\leftrightarrow}},\boldsymbol{\to})$.
For example, if $\mathbf{M}$ is transitive, then its diagram consists just of one point. 
In fact, a finitary $2$-representation is transitive if and only if its diagram consists of one point.

\begin{example}\label{ex1}
{\rm 
For $\mathtt{i}\in\cC$, consider the {\em principal} $2$-representation 
$\mathbf{P}_{\mathtt{i}}:=\cC(\mathtt{i},{}_-)$ and set 
\begin{displaymath}
\mathbf{P}:=\bigoplus_{\mathtt{i}\in\ccC}\mathbf{P}_{\mathtt{i}}.  
\end{displaymath}
Then $\mathrm{Ind}(\mathbf{P})=\mathcal{S}[\cC]$, the preorder $\boldsymbol{\to}$
coincides with $\leq_L$ and thus $\boldsymbol{\leftrightarrow}$ is the partial order
induced by $\leq_L$ on the set of all left cells for $\cC$.
}
\end{example}

\begin{example}\label{ex2}
{\rm 
Let $A$ be a basic, connected and non-simple finite dimensional $\Bbbk$-algebra.
Let $\mathcal{C}$ be a small category equivalent to $A\text{-}\mathrm{mod}$. 
Consider the associated $2$-category $\cC_A$ of {\em projective functors} on $\mathcal{C}$ 
as in \cite[Subsection~7.3]{MM1}. Assume that  $1=e_1+e_2+\dots+e_n$ is a primitive decomposition of $1\in A$,
then $\cC_A$ has exactly $n+1$ left cells: the left cell $\mathcal{L}_0$ containing the identity
$1$-morphism; and the left cells $\mathcal{L}_i$, for $i=1,2,\dots,n$, such that 
$\mathcal{L}_i$ contains the $1$-morphisms corresponding to indecomposable  $A\text{-}A$-bimodules
$Ae_j\otimes_{\Bbbk}e_iA$, where $j\in\{1,2,\dots,n\}$. The diagram of the $2$-representation
$\mathbf{P}$ from Example~\ref{ex1} in this case is:
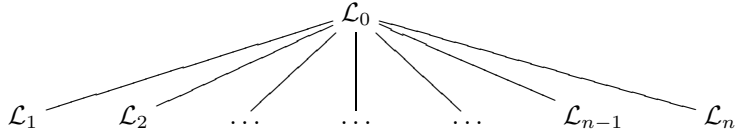
\begin{displaymath}
\xymatrix{ 
&&&\mathcal{L}_0\ar@{-}[dlll]\ar@{-}[dll]\ar@{-}[dl]\ar@{-}[d]\ar@{-}[dr]\ar@{-}[drr]\ar@{-}[drrr]&&&\\
\mathcal{L}_1&\mathcal{L}_2&\dots&\dots&\dots&\mathcal{L}_{n-1}&\mathcal{L}_n
}
\end{displaymath}
Here, the elements in $\mathrm{Ind}(\mathbf{M})/_{\boldsymbol{\leftrightarrow}}$ are identified with
the corresponding simple transitive subquotients in $\mathbf{P}$.
}
\end{example}

\begin{example}\label{ex3}
{\rm 
Consider the $2$-category $\cS_3$ of {\em Soergel bimodules} for the symmetric group
$S_3=\{e,s,t,st,ts,sts=tst\}$, see \cite[Example~3]{MM2} and Subsection~\ref{s6.2}. 
Indecomposable $1$-morphisms
in $\cS_3$ correspond to elements in the {\em Kazhdan-Lusztig basis} of the ring
$\mathbb{Z}S_3$, see \cite{KaLu}, and the preorder $\geq_L$ is just the {\em Kazhdan-Lusztig
left order}. The left cells are given by $\{e\}$, $\{s,ts\}$, $\{t,st\}$, $\{sts\}$.
The diagram of the $2$-representation $\mathbf{P}$ from Example~\ref{ex1} in this case is:
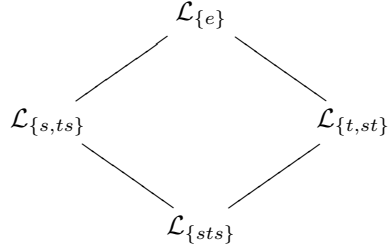
\begin{displaymath}
\xymatrix{ 
&\mathcal{L}_{\{e\}}\ar@{-}[dl]\ar@{-}[dr]&\\
\mathcal{L}_{\{s,ts\}}\ar@{-}[dr]&&\mathcal{L}_{\{t,st\}}\ar@{-}[dl]\\
&\mathcal{L}_{\{sts\}}&
}
\end{displaymath}
Here, again, the elements in $\mathrm{Ind}(\mathbf{M})/_{\boldsymbol{\leftrightarrow}}$ are identified with
the corresponding simple transitive subquotients in $\mathbf{P}$.
}
\end{example}

\subsection{Diagrams and direct sums}\label{s3.3}

Let $\cC$ be a finitary $2$-category. Let $\mathbf{M}$ and $\mathbf{N}$ be finitary $2$-representations 
of $\cC$. Then the diagram of $\mathbf{M}\oplus\mathbf{N}$ is the disjoint union of the
diagrams for $\mathbf{M}$ and $\mathbf{N}$. 

On the other hand, if $\mathbf{K}$ is a finitary 
$2$-representation  of $\cC$ whose diagram is not connected, then, in general, it is not true that 
$\mathbf{K}$ is equivalent to a direct sum of two finitary $2$-representations  of $\cC$.
For example, the inflation (in the sense  of \cite[Subsection~3.6]{MM6}) of a transitive $2$-representation
by a connected finitary category $\mathcal{A}$ cannot be decomposed into a direct sum. On the other
hand, the diagram of such an inflation is a disjoint union of dots whose number equals the number
of isomorphism classes of indecomposable objects in $\mathcal{A}$.

Abusing the language, we will say that the diagram of a $2$-representation is {\em semi-simple} 
provided that it is a disjoint union of dots, that is, it has no edges.

\subsection{Diagrams of isotypic modules}\label{s3.4}

Recall from \cite[Section~4]{MM5} that a finitary $2$-representation $\mathbf{M}$ of a finitary $2$-category 
$\cC$ is called {\em isotypic} provided that all simple transitive subquotients of $\mathbf{M}$ are equivalent
(with each other). The observation from the previous subsection and the proof of 
\cite[Lemma~9]{MM6} imply the following.

\begin{proposition}\label{prop4}
Let $\cC$ be a weakly fiat $2$-category in which all two-sided cells are strongly regular.
Then the diagram of any isotypic $2$-representation of $\cC$ is semi-simple.
\end{proposition}

The arguments used in  \cite[Lemma~9]{MM6} can be generalized in a way 
we explain in the remainder of this subsection. 

A two-sided cell $\mathcal{J}$ of a finitary $2$-category $\cC$ is called {\em good} provided that 
it has the following property:
\begin{itemize}
\item there is an element $x\in B^{(\mathcal{J})}_+$ and $n,k,l\in\{1,2,3,\dots\}$, with $n>k\geq l$, such that 
\begin{equation}\label{eqgood}
x^n+a_{n-1}x^{n-1}+\dots +a_{k+1}x^{k+1}=a_k x^k+a_{k-1} x^{k-1}+\dots+ a_l x^l
\end{equation}
for some non-negative real numbers $a_{n-1},a_{n-2},\dots, a_l$ with $a_l\neq 0$.
\end{itemize}

\begin{example}\label{ex5}
{\rm 
Let $A$ be a basic, connected and non-simple finite dimensional $\Bbbk$-algebra.
Let $\mathcal{C}$ be a small category equivalent to $A\text{-}\mathrm{mod}$. 
Consider the associated $2$-category $\cC_A$ of {\em projective functors} on $\mathcal{C}$ 
as in \cite[Subsection~7.3]{MM1}. Then $\cC_A$ contains two two-sided cells, namely the two-sided cell
$\mathcal{J}_0$ with the identity $1$-morphism and the two-sided cells $\mathcal{J}$ with 
all other indecomposable $1$-morphisms.  Both $\mathcal{J}_0$ and $\mathcal{J}$ are, obviously, idempotent.
The two-sided cell $\mathcal{J}_0$ is good as the unique element of this cell is an idempotent. 
The two-sided cell $\mathcal{J}$ is good as well. Indeed, for $\mathrm{F}$ given as the tensor 
product with $A\otimes_{\Bbbk}A$, we have $\mathrm{F}\circ \mathrm{F}=\mathrm{F}^{\oplus \dim A}$.
}
\end{example}

For a rather different example of a good two-sided cell, see Subsection~\ref{s3.6} below.
We would like to formulate the following conjecture:

\begin{conjecture}\label{conj5}
Let $\cC$ be a finitary $2$-category such that each two-sided cell of $\cC$ is idempotent.
Then the diagram of any isotypic $2$-representation of $\cC$ is semi-simple.
\end{conjecture}

Below we show that Conjecture~\ref{conj5} is true in many cases related to good two-sided cells.

\begin{theorem}\label{thm6}
Let $\cC$ be a finitary $2$-category such that each two-sided cell of $\cC$ is idempotent.  
Let $\mathbf{N}$ be a simple transitive $2$-representation of $\cC$ with a good apex $\mathcal{J}$.
Let $\mathbf{M}$ be an isotypic $2$-representation of $\cC$ in which all simple transitive 
$2$-subquotients are equivalent to $\mathbf{N}$. Then the diagram of $\mathbf{M}$  is semi-simple. 
\end{theorem}

\begin{proof}
Our proof is a generalization of the argument given in \cite[Lemma~9]{MM6}. Let 
$\{\mathrm{F}_i\,:\,i=1,2,\dots,m\}$ be a complete and irredundant list of representative of
the isomorphism classes of indecomposable $1$-morphisms in the cell $\mathcal{J}$. Let 
$\{c_i\,:\,i=1,2,\dots,m\}$ be positive real numbers such that the element
\begin{displaymath}
x=\sum_{i=1}^m c_i[\mathrm{F}_i] 
\end{displaymath}
satisfies 
\begin{displaymath}
x^n+a_{n-1}x^{n-1}+\dots +a_{k+1}x^{k+1}=a_k x^k+a_{k-1} x^{k-1}+\dots+ a_l x^l
\end{displaymath}
in $B_+^{(\mathcal{J})}$ for some non-negative real numbers $a_{n-1},a_{n-2},\dots, a_l$ with $a_l\neq 0$.
These exist as $\mathcal{J}$ is assumed to be good. Set
\begin{displaymath}
x_{\mathbf{N}}:= \sum_{i=1}^m c_i[\mathrm{F}_i]^{\mathbf{N}}_{\oplus}\qquad\text{ and }\qquad
x_{\mathbf{M}}:= \sum_{i=1}^m c_i[\mathrm{F}_i]^{\mathbf{M}}_{\oplus}.
\end{displaymath}
Then we have
\begin{equation}\label{eq3}
x_{\mathbf{N}}^n+a_{n-1}x_{\mathbf{N}}^{n-1}+\dots +a_{k+1}x_{\mathbf{N}}^{k+1}=a_k x_{\mathbf{N}}^k
+a_{k-1} x_{\mathbf{N}}^{k-1}+\dots+ a_l x_{\mathbf{N}}^l
\end{equation}
as $\mathbf{N}$ annihilates all two-sided cells which are strictly larger than $\mathcal{J}$ in the
two-sided order. Note that, by Lemma~\ref{lem801}, all coefficients in the matrix 
$x_{\mathbf{N}}$ are positive.

As explained in the proof of \cite[Lemma~9]{MM6}, it is enough to consider the case when
$\mathbf{M}$ has exactly two simple transitive subquotients (each equivalent to $\mathbf{N}$).
In this case, we may assume that the matrix $x_{\mathbf{M}}$ has the form
\begin{equation}\label{eq5}
x_{\mathbf{M}}=\left(\begin{array}{c|c}x_{\mathbf{N}}&X\\\hline 0&x_{\mathbf{N}}\end{array}\right)
\end{equation}
for some matrix $X$ with non-negative coefficients. To prove our theorem we need to show that $X=0$.

Let $\mathrm{G}$ be an indecomposable $1$-morphism such that $\mathbf{M}(\mathrm{G})\neq 0$. 
We claim that $\mathrm{G}\leq_J \mathcal{J}$. Indeed, assume that this is not the case and 
denote by $\mathcal{J}'$ the two-sided cell containing $\mathrm{G}$. Let $\mathrm{G}_1$, 
$\mathrm{G}_2$ and $\mathrm{G}_3$ be indecomposable $1$-morphism in $\mathcal{J}'$ such that 
$\mathrm{G}_1$ is isomorphic to a direct summand of $\mathrm{G}_2\circ \mathrm{G}_3$. These
exist because $\mathcal{J}'$ is idempotent by our assumptions. As $\mathbf{N}$ annihilates all
$\mathrm{G}_i$, for $i=1,2,3$, we can chose an ordering of indecomposable objects in all
$\mathbf{M}(\mathtt{i})$, where $\mathtt{i}\in\cC$, such that the matrices of all 
$\mathbf{M}(\mathrm{G}_i)$ have the form
\begin{equation}\label{eq1}
\left(\begin{array}{c|c}0&*\\\hline 0&0\end{array} \right) 
\end{equation}
(here the diagonal blocks correspond to simple transitive $2$-subquotients). Now, on the one hand,
all these matrices must be non-zero (as $\mathcal{J}'$ does not belong to the annihilator of $\mathbf{M}$).
On the other hand, the matrix of $\mathbf{M}(\mathrm{G}_2\circ \mathrm{G}_3)=
\mathbf{M}(\mathrm{G}_2)\circ \mathbf{M}(\mathrm{G}_3)$ is zero (because any product of two matrices of 
the form \eqref{eq1} is zero) and hence the matrix of 
$\mathbf{M}(\mathrm{G}_1)$, being a non-negative  summand of the matrix of 
$\mathbf{M}(\mathrm{G}_2\circ \mathrm{G}_3)$, is zero as well, a contradiction.

The previous paragraph implies that 
\begin{equation}\label{eq4}
x_{\mathbf{M}}^n+a_{n-1}x_{\mathbf{M}}^{n-1}+\dots +a_{k+1}x_{\mathbf{M}}^{k+1}=a_k x_{\mathbf{M}}^k
+a_{k-1} x_{\mathbf{M}}^{k-1}+\dots+ a_l x_{\mathbf{M}}^l.
\end{equation}
We plug \eqref{eq5} into \eqref{eq4} and compare the right upper part of the matrices on both sides.
This gives the equation
\begin{multline*}
\left(\sum_{j=0}^{n-1}x_{\mathbf{N}}^{n-1-j}Xx_{\mathbf{N}}^{j}\right)+
a_{n-1}\left(\sum_{j=0}^{n-2}x_{\mathbf{N}}^{n-2-j}Xx_{\mathbf{N}}^{j}\right)+\dots+
a_{k+1}\left(\sum_{j=0}^{k}x_{\mathbf{N}}^{k-j}Xx_{\mathbf{N}}^{j}\right)=\\=
a_{k}\left(\sum_{j=0}^{k-1}x_{\mathbf{N}}^{k-1-j}Xx_{\mathbf{N}}^{j}\right)+
a_{k-1}\left(\sum_{j=0}^{k-2}x_{\mathbf{N}}^{k-2-j}Xx_{\mathbf{N}}^{j}\right)+\dots+
a_{l}\left(\sum_{j=0}^{l-1}x_{\mathbf{N}}^{l-1-j}Xx_{\mathbf{N}}^{j}\right).
\end{multline*}
We multiply both summands with $x_{\mathbf{N}}^{k}$ from the left and also add to both sides
summands 
\begin{displaymath}
a_{k-1}x_{\mathbf{N}}^{k-1}Xx_{\mathbf{N}}^{k-1},\,\,\, 
a_{k-2}x_{\mathbf{N}}^{k-1}Xx_{\mathbf{N}}^{k-2},\,\,\,  
a_{k-2}x_{\mathbf{N}}^{k-2}Xx_{\mathbf{N}}^{k-1},\,\,\, \dots,\,\,\, 
a_{l}x_{\mathbf{N}}^{l}Xx_{\mathbf{N}}^{k-1}. 
\end{displaymath}
Now we can use \eqref{eq3} to simplify both sides of the above equation and in this way get rid  
of all summands on the right hand side  leaving us with an equation of the form
``a linear combination of matrices of the form $x_{\mathbf{N}}^{s}Xx_{\mathbf{N}}^{t}$ with non-negative
integer coefficients equals zero''. As $n>k$, at least one remaining summand on the left hand side
will appear with a positive coefficient, namely the coefficient $1$ coming from the summand $x_{\mathbf{M}}^n$.
This yields the condition $x_{\mathbf{N}}^{s}Xx_{\mathbf{N}}^{t}=0$, for some positive $s$ and $t$.
Since all coefficients of $x_{\mathbf{N}}$ are positive and all coefficients of $X$ are non-negative, 
this forces  $X=0$, which completes the proof.
\end{proof}

An immediate consequence of Theorem~\ref{thm6} is the following corollary.

\begin{corollary}\label{thm6-1}
Let $\cC$ be a finitary $2$-category such that each two-sided cell of $\cC$ is good and idempotent.  
Then the diagram of any isotypic $2$-representation of $\cC$ is semi-simple. 
\end{corollary}

We note that Theorem~\ref{thm6} is interesting in connection to the remark made at the end of 
\cite[Section~4]{Lu}.

\subsection{Soergel bimodules of type $B_2$}\label{s3.6}

Let $\cS$ be a finitary $2$-category of Soergel bimodules over the coinvariant algebra of a
Weyl group of type $B_2$, see \cite{Zi} for details. Indecomposable $1$-morphisms for $\cS$
are $\theta_w$, where $w$ runs through the Weyl group 
\begin{displaymath}
W=\{e,s,t,st,ts,sts,tst,w_0:=stst=tsts\} 
\end{displaymath}
of type $B_2$. These indecomposable $1$-morphisms correspond to elements in the Kazhdan-Lusztig basis
of $\mathbb{Z}W$. We have three two-sided cells: $\mathcal{J}_{\{e\}}$, $\mathcal{J}_{\{w_0\}}$
and the cell $\mathcal{J}$ formed by all remaining elements. They all are idempotent and the first two
are good since $\theta_e^2=\theta_e$ and $\theta_{w_0}^2=\theta_{w_0}^{\oplus 8}$. 

For the two-sided cell $\mathcal{J}$, we have
$\mathcal{J}=\{\theta_s,\theta_{ts},\theta_{sts},\theta_t,\theta_{st},\theta_{tst}\}$. 
Then, modulo the ideal generated by $\theta_{w_0}$, multiplication of elements in $\mathcal{J}$ 
is given by the following table:
{\tiny
\begin{displaymath}
\begin{array}{c||c|c|c|c|c|c}
&\theta_s&\theta_{ts}&\theta_{sts}&\theta_t&\theta_{st}&\theta_{tst}\\
\hline\hline
\theta_s&\theta_s{\oplus}\theta_s&\theta_{sts}\oplus \theta_s&\theta_{sts}{\oplus}\theta_{sts}
&\theta_{st}&\theta_{st}\oplus \theta_{st}&\theta_{st}\\
\hline
\theta_{ts}&\theta_{ts}{\oplus}\theta_{ts}&\theta_{ts}{\oplus}\theta_{ts}&\theta_{ts}{\oplus}\theta_{ts}
&\theta_{tst}{\oplus}\theta_{t}&\theta_{tst}{\oplus}\theta_{tst}{\oplus}
\theta_{t}{\oplus}\theta_{t}&\theta_{tst}{\oplus}\theta_{t}\\
\hline
\theta_{sts}&\theta_{sts}{\oplus}\theta_{sts}&\theta_{sts}\oplus \theta_s&\theta_{s}{\oplus}\theta_{s}
&\theta_{st}&\theta_{st}\oplus \theta_{st}&\theta_{st}\\
\hline
\theta_{t}&\theta_{ts}&\theta_{ts}\oplus \theta_{ts}&\theta_{ts}
&\theta_{t}\oplus \theta_{t}&\theta_{tst}\oplus \theta_{t}&\theta_{tst}\oplus \theta_{tst}\\
\hline
\theta_{st}&\theta_{sts}\oplus \theta_{s}&\theta_{sts}\oplus\theta_{sts}\oplus \theta_{s}
\oplus \theta_{s}&\theta_{sts}\oplus \theta_{s}
&\theta_{st}\oplus \theta_{st}&\theta_{st}\oplus \theta_{st}&\theta_{st}\oplus \theta_{st}\\
\hline
\theta_{tst}&\theta_{ts}&\theta_{ts}\oplus \theta_{ts}&\theta_{ts}
&\theta_{tst}\oplus \theta_{tst}&\theta_{tst}\oplus \theta_{t}&\theta_{t}\oplus \theta_{t}\\
\hline
\end{array} 
\end{displaymath}
}
Consider the element
\begin{displaymath}
x:=[\theta_s]_{\oplus}+[\theta_t]_{\oplus}+[\theta_{sts}]_{\oplus}+
[\theta_{tst}]_{\oplus}+\sqrt{2}[\theta_{st}]_{\oplus}+\sqrt{2}[\theta_{ts}]_{\oplus}\in
B_+^{(\mathcal{J})}.
\end{displaymath}
A direct computation, using the above multiplication table of the $\theta_w$'s,
gives the equality $x^{2}=(8+4\sqrt{2})x$. This implies that the two-sided cell 
$\mathcal{J}$ is good. From Corollary~\ref{thm6-1} we thus get that the diagram of 
any isotypic $2$-representation of $\cS$ is semi-simple.

\subsection{Soergel bimodules of type $I_2(5)$}\label{s3.7}

Let $\cS$ be a finitary $2$-category of Soergel bimodules over the coinvariant algebra of a
Coxeter group of type $I_2(5)$, see \cite{El} for details. Indecomposable $1$-morphisms for $\cS$
are $\theta_w$, where $w$ runs through the dihedral group 
\begin{displaymath}
D_{2\cdot 5}=\{e,s,t,st,ts,sts,tst,stst,tsts,w_0:=ststs=tstst\} 
\end{displaymath}
which is a Coxeter group of type $I_2(5)$. 
These indecomposable $1$-morphisms correspond to elements in the Kazhdan-Lusztig basis
of $\mathbb{Z}W$. We have three two-sided cells: $\mathcal{J}_{\{e\}}$, $\mathcal{J}_{\{w_0\}}$
and the cell $\mathcal{J}$ formed by all remaining elements. They all are idempotent and the first two
are good since $\theta_e^2=\theta_e$ and $\theta_{w_0}^2=\theta_{w_0}^{\oplus 10}$. 

For the two-sided cell $\mathcal{J}$, we have 
$\mathcal{J}=\{\theta_s,\theta_{ts},\theta_{sts},\theta_t,\theta_{st},\theta_{tst},\theta_{stst},\theta_{tsts}\}$. 
Then, modulo the ideal generated by $\theta_{w_0}$, multiplication of elements in $\mathcal{J}$ 
is given by the following table, in which, for simplicity, we denote $\theta_w$ by $w$:
{\tiny
\begin{displaymath}
\begin{array}{c||c|c|c|c|c|c|c|c}
&{}s&{}{ts}&{}{sts}&{}{tsts}&{}t&{}{st}&{}{tst}&{}{stst}\\
\hline\hline
{}s&{}s^{\oplus 2}&{}{sts}\oplus {}s&{sts}^{\oplus 2}
&{}{sts}&st&{st}^{\oplus 2}&{}{stst}\oplus st&stst^{\oplus 2}\\
\hline
{}{ts}&{ts}^{\oplus 2}&{}{tsts}{\oplus}{}{ts}^{\oplus 2}&{}{tsts}^{\oplus 2}{\oplus}{}{ts}^{\oplus 2}
&{}{tsts}{\oplus}{}{ts}&{}{tst}{\oplus}{t}&{}{tst}^{\oplus 2}{\oplus}{}{t}^{\oplus 2}
&{}{tst}^{\oplus 2}{\oplus}{}{t}&tst^{\oplus 2}\\
\hline
{}{sts}&{}{sts}^{\oplus 2}&{}{sts}^{\oplus 2}\oplus {}s&{}{sts}^{\oplus 2}{\oplus}{}{s}^{\oplus 2}
&{}{sts}\oplus s&{}{stst}\oplus {}{st}&{}{stst}^{\oplus 2}\oplus {}{st}^{\oplus 2}
&{}{stst}\oplus {}{st}^{\oplus 2}&{st}^{\oplus 2}\\
\hline
{}{tsts}&{}{tsts}^{\oplus 2}&{}{tsts}\oplus {}ts&{}{ts}^{\oplus 2}
&{}{ts}&{}{tst}&{}{tst}^{\oplus 2}&{}{tst}\oplus {}{t}&{t}^{\oplus 2}\\
\hline
{}{t}&{}{ts}&{ts}^{\oplus 2}&{}{tsts}\oplus ts
&{tsts}^{\oplus 2}&t^{\oplus 2}&{}{tst}\oplus {}{t}&tst^{\oplus 2}&tst\\
\hline
{}{st}&{}{sts}\oplus {}{s}&{}{sts}^{\oplus 2}\oplus {}{s}^{\oplus 2}
&{}{sts}^{\oplus 2}\oplus s&{sts}^{\oplus 2}&{}{st}^{\oplus 2}&{}{stst}\oplus {}{st}^{\oplus 2}
&{}{stst}^{\oplus 2}\oplus {}{st}^{\oplus 2}&{stst}\oplus st\\
\hline
{}{tst}&tsts\oplus {}{ts}&{}{tsts}^{\oplus 2}\oplus {}{ts}^{\oplus 2}&{}{tsts}\oplus ts^{\oplus 2}
&{}{ts}^{\oplus 2}&{}{tst}^{\oplus 2}&{}{tst}^{\oplus 2}\oplus {}{t}
&{}{tst}^{\oplus 2}\oplus {}{t}^{\oplus 2}&tst\oplus t\\
\hline
{}{stst}&sts&{}{sts}^{\oplus 2}&sts\oplus s&s^{\oplus 2}&{}{stst}^{\oplus 2}&{}{stst}\oplus {}{st}
&{}{st}^{\oplus 2}&st\\
\hline
\end{array} 
\end{displaymath}
}
Consider the element
\begin{multline*}
y:=\frac{1}{2}[\theta_s]_{\oplus}+\frac{1}{2}[\theta_t]_{\oplus}+[\theta_{st}]_{\oplus}+[\theta_{ts}]_{\oplus}+\\
+ [\theta_{sts}]_{\oplus}+
[\theta_{tst}]_{\oplus}+\frac{1}{2}[\theta_{stst}]_{\oplus}+\frac{1}{2}[\theta_{tsts}]_{\oplus}\in
B_+^{(\mathcal{J})}.
\end{multline*}
A direct computation, using the above multiplication table of the $\theta_w$'s,
gives the equality $y^{3}=15y^2+5y$. This implies that the two-sided cell 
$\mathcal{J}$ is good. From Corollary~\ref{thm6-1} we thus get that the diagram of 
any isotypic $2$-representation of $\cS$ is semi-simple.

\section{Discrete extensions of finitary $2$-representations}\label{s4}

\subsection{Decorated diagrams}\label{s4.1}

Let $\cC$ be a finitary $2$-category and $\mathbf{M}$ a finitary $2$-representation of $\cC$.
Let, further, $\rho$ and $\tau$ be two elements in $\mathrm{Ind}(\mathbf{M})/_{\boldsymbol{\leftrightarrow}}$
which are connected by an edge in the diagram of $\mathbf{M}$ and such that 
$\rho\boldsymbol{\to}\tau$. In this situation we can decorate the edge
between $\rho$ and $\tau$ by the set $\Theta_{\tau,\rho}$ which consists of the isomorphism
classes of all indecomposable $1$-morphisms $\mathrm{F}$ in $\cC$ for which there is 
$X\in\rho$  such that  $\mathrm{F}\, X$ contains, 
as a direct summand, some element in $\tau$. Note that, by construction, each edge in the 
diagram of $\mathbf{M}$ is decorated by a non-empty set. The obtained construct is called
the {\em decorated} diagram of $\mathbf{M}$. In the classical representation theory, a similar 
objects appear, in particular, in \cite{Ri}.

\begin{example}\label{ex21}
{\rm
In the situation described in Example~\ref{ex2}, the decorated diagram of the  $2$-representation
$\mathbf{P}$ is:
\begin{displaymath}
\xymatrix{ 
&&&\mathcal{L}_0\ar@{-}[dlll]|-{\mathcal{L}_1}
\ar@{-}[dll]|-{\mathcal{L}_2}\ar@{-}[dl]|-{\dots}
\ar@{-}[d]|-{\dots}\ar@{-}[dr]|-{\dots}
\ar@{-}[drr]|-{\mathcal{L}_{n-1}}\ar@{-}[drrr]|-{\mathcal{L}_n}&&&\\
\mathcal{L}_1&\mathcal{L}_2&\dots&\dots&\dots&\mathcal{L}_{n-1}&\mathcal{L}_n
}
\end{displaymath}
} 
\end{example}

\begin{example}\label{ex22}
{\rm
In the situation described in Example~\ref{ex3}, the decorated diagram of the  $2$-representation
$\mathbf{P}$ is:
\begin{displaymath}
\xymatrix{ 
&\mathcal{L}_{\{e\}}\ar@{-}[dl]|-{{\{s,ts\}}}\ar@{-}[dr]|-{{\{t,st\}}}&\\
\mathcal{L}_{\{s,ts\}}\ar@{-}[dr]|-{{\{s,st,ts,sts\}}}&&\mathcal{L}_{\{t,st\}}\ar@{-}[dl]|-{{\{t,st,ts,sts\}}}\\
&\mathcal{L}_{\{sts\}}&
}
\end{displaymath}
} 
\end{example}

Decorated diagrams contain slightly more information than ordinary diagrams. However, we would like to 
note that the information encoded in decorated diagrams does not provide the ``full picture'' of the action
on the level of objects (in the $\mathbf{M}(\mathtt{i})$'s) and $1$-morphisms (in $\cC$). 
The reason for this is the analogy with
usual diagrams of usual modules. Consider the algebra $A=\mathbb{C}[x,y]/(x^3,y^2,xy)$. Let $M$ be the 
quotient of the regular $A$-module by the submodule generated by $y-x^2$. Then $M$ is uniserial of length three
and hence has the following diagram:
\begin{displaymath}
\xymatrix{
\bullet\ar@{-}[d]\\\bullet\ar@{-}[d]\\\bullet 
} 
\end{displaymath}
In analogy with the above, we can decorate both edges of this diagram with $x$. However, the action of $y$
maps the top of $M$ to the socle of $M$ and does not factor through any intermediate subquotients.
Therefore it does not appear on the decorated diagram.

\subsection{Discrete extensions}\label{s4.2}

Let $\cC$ be a finitary $2$-category and $\mathbf{M}$ a finitary $2$-representation of $\cC$.
Now we would like to restrict our consideration to the case when the diagram of $\mathbf{M}$
has exactly two vertices. We denote by $\mathbf{K}$ a transitive $2$-subrepresentations of 
$\mathbf{M}$ and by $\mathbf{N}$ the quotient of $\mathbf{M}$ by the ideal generated by the
identity morphisms for all objects in all $\mathbf{K}(\mathtt{i})$, where $\mathtt{i}\in\cC$.
In this situation we say that 
\begin{equation}\label{eq201}
0\to \mathbf{K}\to\mathbf{M}\to\mathbf{N}\to 0, 
\end{equation}
where $\mathbf{K}\to\mathbf{M}$ is the natural inclusion and $\mathbf{M}\to\mathbf{N}$ is the natural
projection, is a {\em short exact sequence} of $2$-representations. For the corresponding notion in 
the theory of finitary $\Bbbk$-linear categories, see \cite[Subsection~2.1.1]{SVV}. For 
the corresponding notion in the theory of abelian categories, see \cite{PV} and references therein.

In the situation above, the diagram of $\mathbf{M}$ is either
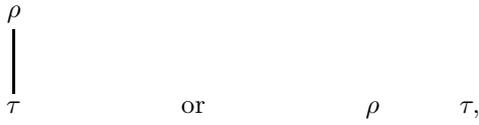
\begin{displaymath}
\xymatrix{
\rho\ar@{-}[d]&&&&&\\
\tau&&\mathrm{or}&&\rho&\tau,
}
\end{displaymath}
where $\rho$ corresponds to $\mathbf{N}$ and $\tau$ corresponds to $\mathbf{K}$.
In the  case of the left diagram, we say that the set $\Theta_{\tau,\rho}$ represents the {\em discrete extension}
of $\mathbf{N}$ by $\mathbf{K}$ corresponding to $\mathbf{M}$. In the case of the right diagram, we say that this
discrete extension is represented by the empty set.

For two transitive finitary $2$-representations $\mathbf{N}'$ and $\mathbf{K}'$, we define 
the (finite) set $\mathrm{Dext}(\mathbf{N}',\mathbf{K}')$ of {\em discrete extensions}
of $\mathbf{N}'$ by $\mathbf{K}'$, as the set of all non-empty subsets $\Theta$ of the set of isomorphism
classes of indecomposable $1$-morphisms in $\cC$ for which there is a short exact sequence \eqref{eq201}
such that $\mathbf{K}$ is equivalent to $\mathbf{K}'$, $\mathbf{N}$ is equivalent to $\mathbf{N}'$,
and $\Theta$ represents the discrete extension of $\mathbf{N}$ by $\mathbf{K}$ corresponding to $\mathbf{M}$.
Directly from the definition we have that 
\begin{displaymath}
\mathrm{Dext}(\mathbf{N}',\mathbf{K}')= \mathrm{Dext}(\mathbf{N}'',\mathbf{K}'')
\end{displaymath}
provided that $\mathbf{K}''$ is equivalent to $\mathbf{K}'$ and $\mathbf{N}''$ is equivalent to $\mathbf{N}'$.

From the definitions, we have that $\mathrm{Dext}(\mathbf{N}',\mathbf{K}')=\varnothing$ is equivalent to 
saying that the diagram of any short exact sequence \eqref{eq201}
such that $\mathbf{K}$ is equivalent to $\mathbf{K}'$ and $\mathbf{N}$ is equivalent to $\mathbf{N}'$
is semi-simple. For example, the statement of Theorem~\ref{thm6}
can be equivalently reformulated using the formula $\mathrm{Dext}(\mathbf{N},\mathbf{N})=\varnothing$.

\subsection{Elementary properties of discrete extensions}\label{s4.3}

Let $\cC$ be a finitary $2$-category.

\begin{proposition}\label{prop24}
Let $\mathbf{K}$ and $\mathbf{N}$ be transitive $2$-representations of $\cC$. Denote by 
$\mathcal{J}_{\mathbf{N}}$ the apex of $\mathbf{N}$ and by $\mathcal{J}_{\mathbf{K}}$ the apex of $\mathbf{K}$.
Then, for any $\Theta\in \mathrm{Dext}(\mathbf{N},\mathbf{K})$ and any $\mathrm{F}\in \Theta$ such that the
two-sided cell containing  $\mathrm{F}$ is idempotent, we have either
$\mathrm{F}\leq_J \mathcal{J}_{\mathbf{K}}$ or $\mathrm{F}\leq_J \mathcal{J}_{\mathbf{N}}$ (or both).
\end{proposition}

\begin{proof}
Consider a short exact sequence of the form \eqref{eq201}.
Let $\mathbf{K}'$ and $\mathbf{N}'$ denote the simple transitive quotients of  $\mathbf{K}$ and $\mathbf{N}$,
respectively. Assume that $\mathrm{F}\not\leq_J \mathcal{J}_{\mathbf{K}}$ and 
$\mathrm{F}\not\leq_J \mathcal{J}_{\mathbf{N}}$. Then $\mathbf{K}'(\mathrm{F})=\mathbf{N}'(\mathrm{F})=0$
by the definition of the apex. This implies that $\mathbf{K}(\mathrm{F})=\mathbf{N}(\mathrm{F})=0$
by Proposition~\ref{prop23}. Therefore $\mathbf{K}(\mathrm{G})=\mathbf{N}(\mathrm{G})=0$ for any
$\mathrm{G}$ which is two-sided equivalent to $\mathrm{F}$. 

Since the two-sided cell containing $\mathrm{F}$ is idempotent, this cell contains $\mathrm{G}_1$, $\mathrm{G}_2$
and $\mathrm{G}_3$ such that $\mathrm{G}_3$ is isomorphic to a direct summand of $\mathrm{G}_1\circ \mathrm{G}_2$.
From the previous paragraph, for any $\mathrm{G}$ which is two-sided equivalent to $\mathrm{F}$, the matrix
of $\mathbf{M}(\mathrm{G})$, for an appropriately chosen ordering of indecomposable objects, has the form
\eqref{eq1}. Specializing this to $\mathrm{G}_1$, $\mathrm{G}_2$ and $\mathrm{G}_3$, we get that the matrix 
of $\mathrm{G}_3$ is zero. Therefore $\mathbf{M}(\mathrm{G}_3)=0$ and thus $\mathbf{M}(\mathrm{F})=0$ as
$\mathrm{G}_3$ is two-sided equivalent to $\mathrm{F}$. This implies the claim of the proposition.
\end{proof}

For weakly fiat categories, the above observation can be substantially strengthened.

\begin{proposition}\label{prop25}
Assume that $\cC$ is weakly fiat.
Let $\mathbf{K}$ and $\mathbf{N}$ be transitive $2$-representations of $\cC$. 
Denote by  $\mathcal{J}_{\mathbf{N}}$ the apex of $\mathbf{N}$ and by 
$\mathcal{J}_{\mathbf{K}}$ the apex of $\mathbf{K}$. Then, for any 
$\Theta\in \mathrm{Dext}(\mathbf{N},\mathbf{K})$ and any $\mathrm{F}\in \Theta$, we have 
$\mathrm{F}\leq_J \mathcal{J}_{\mathbf{K}}$.
\end{proposition}

\begin{proof}
Consider a short exact sequence of the form \eqref{eq201}.  Assume that there is
an indecomposable object $X$ in some $\mathbf{N}(\mathtt{i})$, and indecomposable
$1$-morphism $\mathrm{F}\in\cC$ such that $\mathbf{M}(\mathrm{F})\, X$ has an
indecomposable summand $Y$ in some $\mathbf{K}(\mathtt{j})$. We consider the abelianization
$\overline{\mathbf{M}}$, an indecomposable projective object $P_X$ in
$\overline{\mathbf{M}}(\mathtt{i})$ which corresponds to $X$ and a simple object
$L_Y$ in $\overline{\mathbf{M}}(\mathtt{j})$ which corresponds to $Y$. Then we have
\begin{displaymath}
0\neq 
\mathrm{Hom}_{\overline{\mathbf{M}}(\mathtt{j})}(\overline{\mathbf{M}}(\mathrm{F})\, P_X,L_Y)\cong
\mathrm{Hom}_{\overline{\mathbf{M}}(\mathtt{i})}(P_X,\overline{\mathbf{M}}(\mathrm{F}^*)\, L_Y)
\end{displaymath}
by adjunction. Therefore $\overline{\mathbf{M}}(\mathrm{F}^*)\, L_Y\neq 0$, which implies
$\mathrm{F}^*\leq_J \mathcal{J}_{\mathbf{K}}$ by the definition of the apex of $\mathbf{K}$.
The claim of the proposition now follows from the following lemma.

\begin{lemma}\label{lem25-1}
Let $\cC$ be a weakly fiat $2$-category. Then, for any $1$-morphism $\mathrm{F}$ in $\cC$, we have 
$\mathrm{F}\sim_{J}\mathrm{F}^*$.
\end{lemma}

\begin{proof}
Let $\mathcal{L}$ be the left cell containing $\mathrm{F}$ and $\mathcal{J}$ be the two-sided 
cell containing $\mathrm{F}$. Let $\mathrm{G}$ be the Duflo involution in $\mathcal{L}$,
see \cite[Section~7]{MM6}. Then, by \cite[Section~7]{MM6}, we have $\mathrm{G}^*\in \mathcal{L}$,
in particular, $\mathrm{G}^*\sim_J \mathrm{F}$. Since $\mathrm{F}\sim_L \mathrm{G}$, we have 
$\mathrm{F}^*\sim_R \mathrm{G}^*$ and thus we have $\mathrm{F}\sim_{J}\mathrm{F}^*$.
\end{proof}
\end{proof}

\begin{corollary}\label{cor26}
Assume that $\cC$ is weakly fiat.
Let $\mathbf{K}$ and $\mathbf{N}$ be transitive $2$-representations of $\cC$. 
Denote by  $\mathcal{J}_{\mathbf{N}}$ the apex of $\mathbf{N}$ and by 
$\mathcal{J}_{\mathbf{K}}$ the apex of $\mathbf{K}$. Then, for any non-empty  
$\Theta\in \mathrm{Dext}(\mathbf{N},\mathbf{K})$, we have $\Theta\cap \mathcal{J}_{\mathbf{K}}\neq \varnothing$.
\end{corollary}

\begin{proof}
Consider a short exact sequence of the form \eqref{eq201}.  Assume that there is
an indecomposable object $X$ in some $\mathbf{N}(\mathtt{i})$, and indecomposable
$1$-morphism $\mathrm{F}\in\Theta$ such that $\mathbf{M}(\mathrm{F})\, X$ has an
indecomposable summand $Y$ in some $\mathbf{K}(\mathtt{j})$. Let $\mathrm{G}$
be the direct sum of all $1$-morphisms in $\mathcal{J}_{\mathbf{K}}$. 
Since  $\mathcal{J}_{\mathbf{K}}$ is the apex of $\mathbf{K}$ and the latter 
is a transitive $2$-representation, from Lemma~\ref{lem801} it follows that $\mathrm{G}\, Y\neq 0$.
Hence $\Theta$ contains some direct summand of $\mathrm{G}\circ \mathrm{F}$. Since any such direct
summand must be bigger than or equal to $\mathcal{J}_{\mathbf{K}}$ in the two-sided order, 
the claim follows immediately from Proposition~\ref{prop25}.
\end{proof}

We also note that it is clear from the definitions that identity $1$-morphisms do not belong to 
$\mathrm{Dext}(\mathbf{N},\mathbf{K})$ for any transitive $2$-representations
$\mathbf{N}$ and $\mathbf{K}$.

\section{Discrete extensions for the $2$-category $\cC_A$}\label{s5}

\subsection{Construction of extensions between cell modules}\label{s5.1}

As in Example~\ref{ex2}, we consider a basic, connected, non-simple, finite dimensional  
$\Bbbk$-algebra $A$. Let $\mathcal{C}$ be a small subcategory of $A\text{-}\mathrm{mod}$
such that the inclusion functor to $A\text{-}\mathrm{mod}$ is an equivalence. Consider 
the associated $2$-category $\cC_A$ (with unique object $\mathtt{i}$)
of {\em projective functors} on $\mathcal{C}$ as in 
\cite[Subsection~7.3]{MM1}. Assume that 
\begin{displaymath}
1=e_1+e_2+\dots+e_n
\end{displaymath}
is a primitive decomposition of the identity $1\in A$. 
For $i\in\{1,2,\dots,n\}$, we denote by $L_i$ a simple object in $\mathcal{C}$ corresponding to $e_i$ 
and by $P_i$ some fixed indecomposable projective cover of $L_i$ in $\mathcal{C}$. We additionally assume
that simple $A$-modules are not projective (this is the case, for instance, if $A$ is as above and,
additionally, self-injective).

Consider the stable category $\underline{\mathcal{C}}$. Fix an indecomposable object $M\in\mathcal{C}$ such that 
$\mathrm{End}_{\underline{\mathcal{C}}}(M)\cong \Bbbk$. Let $S(M)\subset\{1,2,\dots,n\}$ be the set of all
$i$ such that $[M:L_i]\neq 0$. We will call $S(M)$ the {\em signature} of $M$.
Denote by $\mathcal{X}=\mathcal{X}_M$ the additive closure 
(in $\mathcal{C}$) of $M$ together with all projective objects in $\mathcal{C}$.
By construction, $\mathcal{X}$ is an finitary $\Bbbk$-linear category and is equipped with the structure of 
a $2$-representation of $\cC_A$ by restricting the corresponding structure from $\mathcal{C}$. 
We denote this $2$-representation of $\cC_A$ by $\mathbf{M}_M$.

Denote by $\mathbf{K}_M$ the $2$-subrepresentation of $\mathbf{M}_M$ given by restricting the action of 
$\cC_A$ to the full subcategory of $\mathcal{X}$ consisting of all projective objects in $\mathcal{C}$.
Let $\mathbf{N}_M$ be the quotient of $\mathbf{M}_M$ modulo the ideal generated by identity morphisms for
all objects in $\mathbf{K}_M$. Then we have a short exact sequence of $2$-representations of $\cC_A$ as follows:
\begin{displaymath}
0\to\mathbf{K}_M\to\mathbf{M}_M\to\mathbf{N}_M\to 0.
\end{displaymath}

Denote by $\mathcal{L}_0$ the left cell of $\cC_A$ containing the identity $1$-morphism. 
For every $i=1,2,\dots,n$, denote by $\mathcal{L}_i$ the left cell of $\cC_A$ containing a projective
functor corresponding to tensoring with $Ae_i\otimes e_iA$.
Set $\mathcal{J}_0:=\mathcal{L}_0$ and $\mathcal{J}:=\mathcal{L}_1\cup\mathcal{L}_2\cup\dots\cup\mathcal{L}_n$.
Then the two-sided cells
$\mathcal{J}_0$ and $\mathcal{J}$ are the only two-sided cells of $\cC_A$ and we have
$\mathcal{J}_0<_J\mathcal{J}$.

\begin{proposition}\label{lemma31}
{\hspace{2mm}}
\begin{enumerate}[$($i$)$]
\item\label{lemma31.1} The $2$-representations  $\mathbf{K}_M$ and $\mathbf{C}_{\mathcal{L}_i}$ are equivalent,
for each $i\in\{1,2,\dots,n\}$.
\item\label{lemma31.2} The $2$-representations  $\mathbf{N}_M$ and $\mathbf{C}_{\mathcal{L}_0}$ are equivalent.
\item\label{lemma31.3} The discrete extension corresponding to $\mathbf{M}_M$ is the union of
all $\mathcal{L}_j$, where $j\in S(M)$.
\end{enumerate}
\end{proposition}

\begin{proof}
By construction, $\mathbf{N}_M(\mathtt{i})$ is equivalent to $\Bbbk\text{-}\mathrm{mod}$, moreover,
$\mathbf{N}_M(\mathrm{F})=0$ for any indecomposable $1$-morphism $\mathrm{F}\in\cC_A$ which is not 
isomorphic to the identity. Mapping the identity $1$-morphism in $\cC_A$ to an indecomposable 
object of $\mathbf{N}_M(\mathtt{i})$, defines, by the universal property \cite[Lemma~3]{MM3}, a $2$-natural
transformation from $\mathbf{P}_{\mathtt{i}}$ to $\mathbf{N}_M$. This $2$-natural transformation factors through 
$\mathbf{C}_{\mathcal{L}_0}$ and induces an equivalence between the latter and $\mathbf{N}_M$
because $\mathrm{End}_{\underline{\mathcal{C}}}(M)\cong \Bbbk$ by our assumptions,
proving claim~\eqref{lemma31.2}.

Similarly to the above, mapping the identity $1$-morphism in $\cC_A$ to any simple object in $\mathcal{C}$, 
defines, by the universal property \cite[Lemma~3]{MM3}, a $2$-natural transformation from 
$\mathbf{P}_{\mathtt{i}}$ to the defining $2$-representation of $\cC_A$ (i.e. the natural $2$-action of 
$\cC_A$ on $\mathcal{C}$). This $2$-natural transformation sends $1$-morphisms in $\mathcal{J}$ to
projective objects in $\mathcal{C}$ and induces an equivalence between $\mathbf{C}_{\mathcal{L}_i}$ and the
$2$-representation of $\cC_A$ on the category of projective objects in $\mathcal{C}$
(cf. \cite[Subsection~6.5]{MM1}). The latter is equivalent to $\mathbf{K}_M$ by construction. 
This proves claim~\eqref{lemma31.1}.

By construction, the object $M$ has simple subquotients $L_j$, where $j\in S(M)$.  
Now the assertion of claim~\eqref{lemma31.3} follows by noting that, 
for $s,t,v\in\{1,2,\dots,n\}$, we have (up to isomorphism) the following:
\begin{displaymath}
Ae_s\otimes e_t A\otimes_A L_v= 
\begin{cases}
0, &  v\neq t;\\
P_s, & v=t.
\end{cases}
\end{displaymath}
This completes the proof.
\end{proof}

\subsection{Discrete extensions between cell $2$-representations for $\cC_A$}\label{s5.2}

For any subset $S\subset \{1,2,\dots,n\}$, set
\begin{displaymath}
E_S:=\bigcup_{j\in S}\mathcal{L}_j .
\end{displaymath}
Denote by $\mathcal{E}_A$ the set of all  $E_S$, where $S$ runs through the set of all non-empty 
subsets of $\{1,2,\dots,n\}$.

\begin{theorem}\label{thm32}
{\hspace{2mm}}
\begin{enumerate}[$($i$)$]
\item\label{thm32.1} We have $\mathrm{Dext}(\mathbf{C}_{\mathcal{L}_0},\mathbf{C}_{\mathcal{L}_0})=\varnothing$.
\item\label{thm32.2} We have $\mathrm{Dext}(\mathbf{C}_{\mathcal{L}_1},\mathbf{C}_{\mathcal{L}_1})=\varnothing$.
\item\label{thm32.3} We have $\mathrm{Dext}(\mathbf{C}_{\mathcal{L}_1},\mathbf{C}_{\mathcal{L}_0})=\varnothing$
provided that $A$ is self-injective.
\item\label{thm32.4} We have 
$\mathrm{Dext}(\mathbf{C}_{\mathcal{L}_0},\mathbf{C}_{\mathcal{L}_1})\subset \mathcal{E}_A$, moreover,
for any $M\in\mathcal{C}$ such that $\mathrm{End}_{\underline{\mathcal{C}}}(M)\cong \Bbbk$,
we have $E_{S(M)}\in \mathrm{Dext}(\mathbf{C}_{\mathcal{L}_0},\mathbf{C}_{\mathcal{L}_1})$.
\end{enumerate}
\end{theorem}

\begin{proof}
Claim~\eqref{thm32.1} follows directly from the fact that $\mathcal{J}$ is idempotent and
the only indecomposable $1$-morphism which does not annihilate $\mathbf{C}_{\mathcal{L}_0}$ is the
identity $1$-morphism (up to isomorphism). Claim~\eqref{thm32.2} follows from \cite[Lemma~9]{MM6}
(note that the proof of \cite[Lemma~9]{MM6} does not use the general assumption of \cite{MM6}
that $\cC_A$ is weakly fiat, cf. Theorem~\ref{thm6}). In case $A$ is self-injective, $\cC_A$ is weakly fiat, see 
\cite[Subsection~2.8]{MM6}. Therefore claim~\eqref{thm32.3} follows from Propositions~\ref{prop24} and
\ref{prop25}. It remains to prove claim~\eqref{thm32.4}.

Any element in $\mathrm{Dext}(\mathbf{C}_{\mathcal{L}_0},\mathbf{C}_{\mathcal{L}_1})$ is a subset of 
$\mathcal{J}$ by definition. By Proposition~\ref{lemma31}, 
for any $M\in\mathcal{C}$ such that $\mathrm{End}_{\underline{\mathcal{C}}}(M)\cong \Bbbk$,
we have $E_{S(M)}\in \mathrm{Dext}(\mathbf{C}_{\mathcal{L}_0},\mathbf{C}_{\mathcal{L}_1})$. Therefore,
to prove claim~\eqref{thm32.4}, it remains to show that any element in 
$\mathrm{Dext}(\mathbf{C}_{\mathcal{L}_0},\mathbf{C}_{\mathcal{L}_1})$ is a union on $\mathcal{L}$-classes.

For $s,t\in\{1,2,\dots,n\}$, denote by $\mathrm{F}_{st}$ an indecomposable $1$-morphism in $\cC_A$
corresponding to tensoring with $Ae_s\otimes e_t A$. Then, for $s,t,u,v\in\{1,2,\dots,n\}$, we have 
\begin{equation}\label{eq6}
Ae_s\otimes e_t A\otimes_A Ae_u\otimes e_v A\cong 
Ae_s\otimes e_v A^{\oplus \mathrm{dim}(e_t Ae_u)}
\end{equation} 
which implies $\mathrm{F}_{st}\circ \mathrm{F}_{uv}\cong \mathrm{F}_{sv}^{\oplus \mathrm{dim}(e_t Ae_u)}$.
Further, for $s,t\in\{1,2,\dots,n\}$, we have
\begin{equation}\label{eq7}
Ae_s\otimes e_t A\otimes_A P_t\cong P_s^{\oplus \mathrm{dim}(e_t Ae_t)},\quad\text{ that is }\quad
\mathrm{F}_{st}\, P_t\cong P_s^{\oplus \mathrm{dim}(e_t Ae_t)}.
\end{equation} 
Let 
\begin{displaymath}
0\to\mathbf{C}_{\mathcal{L}_1}\to\mathbf{M}\to\mathbf{C}_{\mathcal{L}_0}\to 0 
\end{displaymath}
be a short exact sequence of $2$-representations. 
By the argument in the penultimate paragraph of the proof of Proposition~\ref{lemma31},
we may identify $\mathbf{C}_{\mathcal{L}_1}$ with the defining $2$-representation for $\cC_A$.
Let $X$ be an indecomposable object which is nonzero in 
$\mathbf{C}_{\mathcal{L}_0}(\mathtt{i})$ and $s,t\in\{1,2,\dots,n\}$ be such that 
$\mathrm{F}_{st}\, X$ has a non-zero indecomposable direct summand in $\mathbf{C}_{\mathcal{L}_1}$, say $P_j$,
for some $j\in\{1,2,\dots,n\}$. Applying $\mathrm{F}_{ij}$, where $i\in \{1,2,\dots,n\}$, 
and using \eqref{eq7}, we get that 
$(\mathrm{F}_{ij}\circ \mathrm{F}_{st})\, X$ has a non-zero direct summand in $\mathbf{C}_{\mathcal{L}_1}$
isomorphic to $P_i$. In particular, $\mathrm{F}_{ij}\circ \mathrm{F}_{st}$ is non-zero. Therefore,
from \eqref{eq6} we obtain that $\mathrm{F}_{it}\, X$ has a non-zero direct summand in $\mathbf{C}_{\mathcal{L}_1}$
isomorphic to $P_i$. This means exactly that the whole left cell of $\mathrm{F}_{st}$ belongs to the 
discrete extension in question. This completes the proof.
\end{proof}

After Theorem~\ref{thm32}, it is natural to ask whether 
$\mathrm{Dext}(\mathbf{C}_{\mathcal{L}_0},\mathbf{C}_{\mathcal{L}_1})=\mathcal{E}_A$.

\subsection{Construction of $2$-representations with some diagrams}\label{s5.3}

In this subsection we assume that $A$ is self-injective, in particular, $\cC_A$ is weakly fiat.
The next statement shows that decorations for decorated diagrams for finitary 
$2$-representations of $\cC_A$ are always unions of left cells. 

\begin{proposition}\label{prop92}
Assume that $A$ is self-injective and let $\mathbf{M}$ be a finitary $2$-rep\-re\-sen\-ta\-ti\-on of $\cC_A$.
Then the diagram of $\mathbf{M}$ is a bipartite graph where vertices in $\Gamma_0^{(0)}$ correspond
to simple transitive subquotients equivalent to $\mathbf{C}_{\mathcal{L}_0}$ and vertices in 
$\Gamma_0^{(1)}$ correspond to simple transitive subquotients equivalent to $\mathbf{C}_{\mathcal{L}_1}$.
Moreover, in the decorated diagram of $\mathbf{M}$ each edge is decorated by a union
of left cells from $\mathcal{J}$.
\end{proposition}

\begin{proof}
This is proved by similar arguments as the ones used in the proof of Theorem~\ref{thm32}. 
\end{proof}

In the remainder of this section we show that any bipartite graph is possible as a diagram
for some $2$-representation of $\cC_A$.

Consider a bipartite graph $\Gamma=(\Gamma_0^{(0)},\Gamma_0^{(1)},\Gamma_1)$, where
$\Gamma_0^{(0)}$ and $\Gamma_0^{(1)}$ are disjoint sets of vertices and $\Gamma_1$ is the set of
edges such that each edge connects a vertex from $\Gamma_0^{(0)}$ with a vertex from $\Gamma_0^{(1)}$.
We identify $\Gamma_1$ with the corresponding subset in $\Gamma_0^{(0)}\times \Gamma_0^{(1)}$
and also set $\Gamma_0:=\Gamma_0^{(0)}\cup \Gamma_0^{(1)}$.
Let $\eta$ be a map from $\Gamma_1$ to $\{1,2,\dots,n\}$ such that 
$\eta((v_1,w))=\eta((v_2,w))$, for all $v_1,v_2\in \Gamma_0^{(0)}$ and $w\in \Gamma_0^{(1)}$ 
such that $(v_1,w),(v_2,w)\in \Gamma_1$.

Consider now a quiver $\boldsymbol{\Gamma}$ obtained from $\Gamma$ in the following way: it has
the same vertices as $\Gamma$ and each edge $(v,w)\in \Gamma_1$ is replaced  in 
$\boldsymbol{\Gamma}$  by an arrow  $\alpha_{(v,w)}$. 
For example, if $\Gamma$ has the following form, where the top line lists vertices from $\Gamma_0^{(0)}$ and the
bottom line lists vertices from $\Gamma_0^{(1)}$,
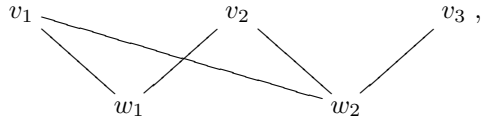
\begin{displaymath}
\xymatrix{
v_1\ar@{-}[dr]\ar@{-}[drrr]&&v_2\ar@{-}[dr]\ar@{-}[dl]&&v_3\ar@{-}[dl]\\
&w_1&&w_2&
},
\end{displaymath}
then $\boldsymbol{\Gamma}$ is the following quiver:
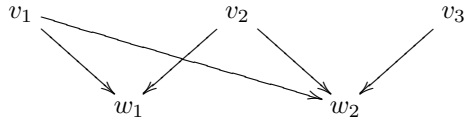
\begin{displaymath}
\xymatrix{
v_1\ar[dr]\ar[drrr]&&
v_2\ar[dr]\ar[dl]&&v_3\ar[dl]\\
&w_1&&w_2&
}
\end{displaymath}

Let $\mathcal{P}$ be a small category equivalent to the category of projective $\Bbbk\boldsymbol{\Gamma}$-modules.
Consider the inflation $\mathbf{M}:=\mathbf{P}^{\boxtimes\mathcal{P}}$ of the principal $2$-representation 
$\mathbf{P}$ of $\cC_A$ by $\mathcal{P}$ and also the abelianization  $\overline{\mathbf{M}}$ of $\mathbf{M}$.
Choose representatives $L_{i,j,x}$ and $L_{0,x}$, where $i,j\in\{1,2,\dots,n\}$ and $x\in \Gamma_0$, 
in the isomorphism classes of simple objects in $\overline{\mathbf{M}}$ 
with the obvious indexing. For $x\in \Gamma_0$
and $i\in\{1,2,\dots,n\}$, denote by $\Delta(i,x)$ the unique, up to isomorphism, indecomposable 
object  such that there is a short exact sequence
\begin{equation}\label{eq11-1}
0\to L_{i,i,x}\to \Delta(i,x)\to  L_{0,x}\to 0.
\end{equation}
Existence and uniqueness of $\Delta(i,x)$ reflects the fact that the top of the indecomposable $A$-$A$--bimodule
${}_AA_A$ is isomorphic to $L_{1,1}\oplus L_{2,2}\oplus \dots\oplus L_{n,n}$, where, for 
$i,j\in\{1,2,\dots,n\}$, we denote by $L_{i,j}$ the simple top of $Ae_i\otimes e_jA$.
As $\cC_A$ is weakly fiat, we also have the dual to \eqref{eq11-1} short exact sequence
\begin{equation}\label{eq11-3}
0\to L_{0,x}\to \nabla(i,x)\to  L_{i,i,x}\to 0.
\end{equation}
We note that the set of all sequences of the form \eqref{eq11-1} is dual to the set of all sequences of
the form \eqref{eq11-3}, however, for each individual $i$, the dual sequence to the sequence 
\eqref{eq11-1} is of the form \eqref{eq11-3}, but does not necessarily correspond to the same $i$.

For $(v,w)\in \Gamma_1$, the arrow $\alpha_{(v,w)}$ naturally defines
an indecomposable object $N_{(v,w)}$ which fits into a non-split short exact sequence
\begin{displaymath}
0\to L_{0,w} \to N_{(v,w)}\to L_{0,v}\to 0.
\end{displaymath}
Consider the direct sum of these short exact sequences over all $w\in \Gamma_0^{(1)}$ for which 
$(v,w)\in \Gamma_1$. Then the diagonal copy of $L_{0,v}$ 
gives rise to the short exact sequence
\begin{equation}\label{eq11}
0\to \bigoplus_{w:(v,w)\in \Gamma_1} L_{0,w} \to N_{v}\to L_{0,v}\to 0.
\end{equation}
For each $L_{0,w}$ in \eqref{eq11}, consider the corresponding short exact sequence of the form
\eqref{eq11-3} and take the direct sum of such sequences over all $w\in \Gamma_0^{(1)}$ such that 
$(v,w)\in \Gamma_1$  to get the short exact sequence 
\begin{displaymath}
\bigoplus_{w:(v,w)\in \Gamma_1} L_{0,w}\hookrightarrow
\bigoplus_{w:(v,w)\in \Gamma_1} \nabla(\eta((v,w)),w)\tto 
\bigoplus_{w:(v,w)\in \Gamma_1} L_{\eta((v,w)),\eta((v,w)),w}.
\end{displaymath}
Write the latter as $0\to X\to Y\to Z\to 0$. Applying the functor
$\mathrm{Hom}(Z,{}_-)$ to \eqref{eq11}, we obtain
\begin{displaymath}
\mathrm{Ext}^1(Z,X)\hookrightarrow  \mathrm{Ext}^1(Z,N_v).
\end{displaymath}
Denote by  $M_v$ the image of $Y$ in $\mathrm{Ext}^1(Z,N_v)$. Then $M_v$ is indecomposable
and fits into a short exact sequence
\begin{displaymath}
0\to \bigoplus_{v:(v,w)\in \Gamma_1} \nabla(\eta((v,w)),w)\to M_{v}\to L_{0,v} \to 0. 
\end{displaymath}
It is easy to see that $M_v$ has trivial endomorphism algebra. Indeed, by construction, $M_v$
is killed by the second power of the radical. Furthermore, again by construction, all simple 
subquotients in the top of $M_v$ are non-isomorphic, moreover, they are not isomorphic to any
of the simple subquotients in the socle of $M_v$.

Denote by $\mathcal{X}$ the additive closure in $\overline{\mathbf{M}}(\mathtt{i})$ of all
$M_v$, where $v\in \Gamma_0^{(0)}$, and all objects of the form 
$\mathrm{F}_{ji}\cdot L_{i,i,w}$, where $w\in \Gamma_0^{(1)}$, $j=1,2,\dots,n$
and $i\in\eta(\Gamma_1)$.

\begin{proposition}\label{prop91}
{\hspace{2mm}} 

\begin{enumerate}[$($i$)$]
\item\label{prop91.1} The category $\mathcal{X}$ is stable under the action of $\cC_A$ and we denote by 
$\mathbf{N}$ the corresponding finitary $2$-representation of $\cC_A$.
\item\label{prop91.2} The diagram of $\mathcal{X}$ is given by $\Gamma$.
\item\label{prop91.3} The decorated diagram of $\mathcal{X}$ is obtained if one decorates each 
$(v,w)\in\Gamma_1$ by the $\mathcal{L}$-class $\mathcal{L}_{\eta((v,w))}$.
\end{enumerate}
\end{proposition}

\begin{proof}
As $A$ is self-injective, $\cC_A$ is weakly fiat, in particular $\cC_A$ always acts by exact functors.
Therefore claim~\eqref{prop91.1} follows directly from the multiplication formula
\begin{displaymath}
Ae_j\otimes e_i A\otimes_A Ae_s\otimes e_t A\cong
Ae_j\otimes e_t A^{\oplus \dim(e_iAe_s)}.
\end{displaymath}
Provided claim~\eqref{prop91.2}, this also implies claim~\eqref{prop91.3}. 

To prove claim~\eqref{prop91.2}, we note that, for all $i,j\in\{1,2,\dots,n\}$ and $v\in\Gamma_0^{(0)}$, we have 
\begin{displaymath}
\mathrm{F}_{ji}\cdot M_v\cong
\bigoplus_{w:(v,w)\in\Gamma_1} \mathrm{F}_{j  \eta((v,w)) }\, L_{\eta((v,w)),\eta((v,w)),w}
\end{displaymath}
by our construction. From the original construction of cell $2$-representations in \cite[Section~4]{MM1}, 
we have that the additive closure of  
\begin{displaymath}
\{\mathrm{F}_{j\eta((v,w))}\,L_{\eta((v,w)),\eta((v,w)),w}\,:\, j\in\{1,2,\dots,n\}\} 
\end{displaymath}
is equivalent to $\mathbf{C}_{\mathcal{L}_{\eta((v,w))}}$. Furthermore, every  $M_v$ corresponds to a 
simple transitive subquotient  equivalent to $\mathbf{C}_{\mathcal{L}_0}$ (since the endomorphism
algebra of this $M_v$ is trivial). This completes the proof.
\end{proof}

\section{Further results and examples}\label{s6}

\subsection{Discrete extensions and the two-sided order}\label{s6.1}

The following generalizes, in some sense, the assertions of Theorem~\ref{thm32}\eqref{thm32.1}-\eqref{thm32.3}.

\begin{theorem}\label{thm41}
Let $\cC$ be weakly fiat. Consider two transitive $2$-representations $\mathbf{K}$ and $\mathbf{N}$ of $\cC$ 
and denote by  $\mathcal{J}_{\mathbf{N}}$ the apex of $\mathbf{N}$ and by 
$\mathcal{J}_{\mathbf{K}}$ the apex of $\mathbf{K}$. Assume that:
\begin{enumerate}[$($a$)$]
\item\label{thm41.1} $\mathcal{J}_{\mathbf{K}}$ is good,
\item\label{thm41.2} $\mathcal{J}_{\mathbf{K}}\leq_J \mathcal{J}_{\mathbf{N}}$,
\item\label{thm41.3} for any left cell $\mathcal{L}\in \mathcal{J}_{\mathbf{N}}$, 
there is a left cell $\mathcal{L}'\in \mathcal{J}_{\mathbf{K}}$
such that $\mathcal{L}'\leq_L\mathcal{L}$.
\end{enumerate}
Then we have $\mathrm{Dext}(\mathbf{N},\mathbf{K})=\varnothing$.
\end{theorem}

\begin{proof}
Consider a short exact sequence of the form \eqref{eq201}. For any $1$-morphism $\mathrm{F}$ in $\cC$,
we have
\begin{displaymath}
[\mathrm{F}]_{\oplus}^{\mathbf{M}}=
\left(\begin{array}{cc}[\mathrm{F}]_{\oplus}^{\mathbf{K}}&Y_{\mathrm{F}}\\
0&[\mathrm{F}]_{\oplus}^{\mathbf{N}}\end{array}\right), 
\end{displaymath}
where all involved matrices have non-negative integer entries.
If $\mathrm{F}$ is an indecomposable $1$-morphism such that $\mathrm{F}\not\leq_J \mathcal{J}_{\mathbf{K}}$,
then condition~\eqref{thm41.2} implies $[\mathrm{F}]_{\oplus}^{\mathbf{K}}=0$. Moreover, in this case we also have 
$Y_{\mathrm{F}}=0$ by arguments similar to those in the proof of Proposition~\ref{prop25}.
If $\mathrm{F}=\mathrm{F}(\mathcal{J}_{\mathbf{K}})$, then $[\mathrm{F}]_{\oplus}^{\mathbf{K}}$
has positive entries.

By condition~\eqref{thm41.1}, $\mathcal{J}_{\mathbf{K}}$ is good. Let $x\in B_+^{(\mathcal{J}_{\mathbf{K}})}$ 
be the corresponding element, 
\begin{displaymath}
x=\sum_{\mathrm{G}\in \mathcal{J}_{\mathbf{K}}}c_{\mathrm{G}}[\mathrm{G}],
\end{displaymath}
where all $c_{\mathrm{G}}>0$. For any $2$-representation $\mathbf{L}$ of $\cC$, we set 
\begin{displaymath}
x_{\mathbf{L}}:=\sum_{\mathrm{G}\in \mathcal{J}_{\mathbf{K}}}c_{\mathrm{G}}[\mathrm{G}]_{\oplus}^{\mathbf{L}}. 
\end{displaymath}
In particular, we have 
\begin{displaymath}
x_{\mathbf{M}}=
\left(\begin{array}{cc}x_{\mathbf{K}}&Y_{x}\\
0&x_{\mathbf{N}}\end{array}\right), 
\end{displaymath}
for some matrix $Y_x$ with non-negative real coefficients.

Let $f(t)$ and $g(t)$ denote the polynomials in the left hand side and the 
right hand side of Equation~\eqref{eqgood}, respectively. In other words, Equation~\eqref{eqgood}
reads $f(x)=g(x)$. Then, from the previous two paragraphs, we have the equality
\begin{displaymath}
\left(\begin{array}{cc}f(x_{\mathbf{K}})&Y\\
0&f(x_{\mathbf{N}})+Q\end{array}\right)=
\left(\begin{array}{cc}g(x_{\mathbf{K}})&Y'\\
0&g(x_{\mathbf{N}})+Q'\end{array}\right), 
\end{displaymath}
where $Q$ and $Q'$ record contributions of $1$-morphisms $\mathrm{H}$ such that 
$\mathrm{H}\not\leq_J \mathcal{J}_{\mathbf{K}}$ and
\begin{gather*}
Y=\sum_{i=0}^{n-1}x_{\mathbf{K}}^{n-1-i}Y_x x_{\mathbf{N}}^i+
a_{n-1}\sum_{i=0}^{n-2}x_{\mathbf{K}}^{n-2-i}Y_x x_{\mathbf{N}}^i+\dots+
a_{k+1}\sum_{i=0}^{k}x_{\mathbf{K}}^{k-i}Y_x x_{\mathbf{N}}^i
\\ 
Y'=a_k\sum_{i=0}^{k-1}x_{\mathbf{K}}^{k-1-i}Y_x x_{\mathbf{N}}^i+
a_{k-1}\sum_{i=0}^{k-2}x_{\mathbf{K}}^{k-2-i}Y_x x_{\mathbf{N}}^i+\dots+
a_{l}\sum_{i=0}^{l-1}x_{\mathbf{K}}^{l-1-i}Y_x x_{\mathbf{N}}^i. 
\end{gather*}
Similarly to the proof of Theorem~\ref{thm6}, the equality $Y=Y'$ implies
$x_{\mathbf{K}}^iY_x x_{\mathbf{N}}^j=0$, for some $i,j>0$. As the matrix 
$x_{\mathbf{K}}$ has positive entries, we get $Y_x x_{\mathbf{N}}^j=0$, for some $j>0$.

By Corollary~\ref{cor26}, to prove the  assertion of the theorem, we need to show that $Y_x=0$. 
For this it would be sufficient to show that each row of $x_{\mathbf{N}}^j$ has a  non-zero entry. 
For the latter, it is enough to show that each row and each column of $x_{\mathbf{N}}$ has
a non-zero entry.

Note that, due to existence of $*$ and Lemma~\ref{lem25-1}, condition~\eqref{thm41.3} 
is equivalent to the condition that,  for any right cell $\mathcal{R}\in \mathcal{J}_{\mathbf{N}}$, 
there is a right cell $\mathcal{R}'\in \mathcal{J}_{\mathbf{K}}$ such that $\mathcal{R}'\leq_R\mathcal{R}$.

Let us first assume that $x_{\mathbf{N}}$ has a zero row. Then the matrix of 
$\mathbf{N}(\mathrm{F}(\mathcal{J}_{\mathbf{K}}))$ has a zero row.
Consider $\mathrm{F}(\mathcal{J}_{\mathbf{N}})$.
The previous paragraph means that $\mathrm{F}(\mathcal{J}_{\mathbf{N}})$ is isomorphic
to a direct summand of $\mathrm{F}(\mathcal{J}_{\mathbf{K}})\circ {\mathrm{G}}$ for some $1$-morphism
${\mathrm{G}}$. Therefore, each entry in the matrix of 
$\mathbf{N}(\mathrm{F}(\mathcal{J}_{\mathbf{N}}))$ is, on the one hand,
non-negative, but, on the other hand, does not exceed the corresponding entry in the matrix of 
$\mathbf{N}(\mathrm{F}(\mathcal{J}_{\mathbf{K}})\circ {\mathrm{G}})$. 
As the matrix of $\mathbf{N}(\mathrm{F}(\mathcal{J}_{\mathbf{K}}))$ has a zero row, 
the matrix of $\mathbf{N}(\mathrm{F}(\mathcal{J}_{\mathbf{K}})\circ {\mathrm{G}})$
has a zero row. Consequently, the matrix of 
$\mathbf{N}(\mathrm{F}(\mathcal{J}_{\mathbf{N}}))$ has a zero row. This contradicts
our assumption that $\mathbf{N}$ is a transitive $2$-representation with apex $\mathcal{J}_{\mathbf{N}}$.

Let us now assume that $x_{\mathbf{N}}$ has a zero column. Then the matrix of 
$\mathbf{N}(\mathrm{F}(\mathcal{J}_{\mathbf{K}}))$ has a zero column.
Consider $\mathrm{F}(\mathcal{J}_{\mathbf{N}})$.
Condition~\eqref{thm41.3} implies that $\mathrm{F}(\mathcal{J}_{\mathbf{N}})$ is isomorphic
to a direct summand of $ \mathrm{G}\circ \mathrm{F}(\mathcal{J}_{\mathbf{K}})$ for some $1$-morphism
${\mathrm{G}}$. Therefore, similarly to the arguments of the previous paragraph, we obtain that
the matrix
$\mathbf{N}(\mathrm{F}(\mathcal{J}_{\mathbf{N}}))$ has a zero column. This contradicts
our assumption that $\mathbf{N}$ is a transitive $2$-rep\-re\-sen\-ta\-ti\-on with apex $\mathcal{J}_{\mathbf{N}}$.
The proof is now complete.
\end{proof}

\subsection{Hidden discrete extensions for Soergel bimodules of type $A_2$}\label{s6.2}

Consider the Weyl group 
\begin{displaymath}
W=\{e,s,t,st,ts,w_0:=sts=tst\} 
\end{displaymath}
of type $A_2$ (it is isomorphic to the symmetric group $S_3$) and denote by $\mathtt{C}$
the corresponding coinvariant algebra. Let $\mathcal{C}$ be a small category equivalent 
to $\mathtt{C}\text{-}\mathrm{mod}$. Let $\cS$ be the finitary $2$-category of Soergel 
bimodules associated to $\mathcal{C}$. This $2$-ca\-te\-go\-ry has:
\begin{itemize}
\item one object $\mathtt{i}$ which can be identified with $\mathcal{C}$;
\item as $1$-morphisms, all endofunctors of $\mathcal{C}$ isomorphic to direct sums of endofunctors
given by tensoring with Soergel $\mathtt{C}\text{-}\mathtt{C}$--bimodules;
\item as $2$-morphisms, all natural transformations of functors.
\end{itemize}
Representatives of isomorphism classes of indecomposable 
$1$-morphisms for $\cS$ are denoted by $\theta_w$, where $w\in W$. These indecomposable 
$1$-morphisms correspond to elements in the Kazhdan-Lusztig basis of $\mathbb{Z}W$. 
Multiplication of elements is given by the following table, where, for simplicity, we
abbreviate $\theta_w$ by $w$:
\begin{displaymath}
\begin{array}{c||c|c|c|c|c|c}
&e&s&t&st&ts&sts\\
\hline\hline
e&e&s&t&st&ts&sts\\
\hline
s&s&s^{\oplus 2}&st&st^{\oplus 2}&sts\oplus s&sts^{\oplus 2}\\
\hline
t&t&ts&t^{\oplus 2}&tst\oplus t&ts^{\oplus 2}&sts^{\oplus 2}\\
\hline
st&st&sts\oplus s&st^{\oplus 2}&sts^{\oplus 2}\oplus st&sts^{\oplus 2}\oplus s^{\oplus 2}&sts^{\oplus 4}\\
\hline
ts&ts&ts^{\oplus 2}&tst\oplus t&sts^{\oplus 2}\oplus t^{\oplus 2}&sts^{\oplus 2}\oplus ts&sts^{\oplus 4}\\
\hline
sts&sts&sts^{\oplus 2}&sts^{\oplus 2}&sts^{\oplus 4}&sts^{\oplus 4}&sts^{\oplus 6}\\
\hline
\end{array} 
\end{displaymath}
We thus have four left cells, namely, $\mathcal{L}_{e}:=\{\theta_e\}$, 
$\mathcal{L}_{s}:=\{\theta_s,\theta_{ts}\}$, $\mathcal{L}_{t}:=\{\theta_t,\theta_{st}\}$
and $\mathcal{L}_{sts}:=\{\theta_{sts}\}$.

The algebra $\mathtt{C}$ is the quotient of the polynomial algebra $\mathbb{C}[x,y,z]$ 
modulo the ideal $I$ generated by the symmetric polynomials $x+y+z$, $xy+xz+yz$ and $xyz$. Here we assume that 
$s$ acts on $\mathbb{C}[x,y,z]$ by swapping $x$ and $y$, and $t$ acts by swapping $y$ and $z$.
We fix the basis in $\mathtt{C}$ consisting of the images, in the quotient $\mathbb{C}[x,y,z]/I$, of the monomials
$1$, $x$, $y$, $x^2$, $xy$ and $x^2y$. In this basis the multiplication in $\mathtt{C}$ is given by the 
following table:
\begin{displaymath}
\begin{array}{c||c|c|c|c|c|c}
&1&x&y&x^2&xy&x^2y\\
\hline\hline
1&1&x&y&x^2&xy&x^2y\\
\hline
x&x&x^2&xy&0&x^2y&0\\
\hline
y&y&xy&-xy-x^2&x^2y&-x^2y&0\\
\hline
x^2&x^2&0&x^2y&0&0&0\\
\hline
xy&xy&x^2y&-x^2y&0&0&0\\
\hline
x^2y&x^2y&0&0&0&0&0\\
\hline
\end{array} 
\end{displaymath}
The subalgebra $\mathtt{C}^s$ of $\mathtt{C}$, consisting of all $s$-invariant elements, has the basis
$1$, $x+y$ and $xy$ and is isomorphic to $\mathbb{C}[q]/(q^3)$. The subalgebra $\mathtt{C}^t$ 
of $\mathtt{C}$, consisting of all $t$-invariant elements, has the basis $1$, $x$ and $x^2$ and 
is also isomorphic to $\mathbb{C}[q]/(q^3)$. For computations, the following table of images, 
in the quotient $\mathbb{C}[x,y,z]/I$, of all other monomials of degree at most $3$ in $\mathbb{C}[x,y,z]$, 
written with respect to the above basis in  $\mathtt{C}$, is very useful:
\begin{gather*}
\begin{array}{l||c|c|c|c|c|c|c|c|c|c|c|c|c|c}
\text{in }\mathbb{C}[x,y,z]&z&y^2&z^2&xz&yz\\
\hline
\text{in }\mathtt{C}&-x-y&-x^2-xy&xy&-x^2-xy&x^2\\
\end{array} \\
\begin{array}{l||c|c|c|c|c|c|c|c|c|c|c|c|c|c}
\text{in }\mathbb{C}[x,y,z]&x^3&y^3&z^3&x^2z&xy^2&xz^2&y^2z&yz^2\\
\hline
\text{in }\mathtt{C}&0&0&0&-x^2y&-x^2y&x^2y&x^2y&-x^2y\\
\end{array} 
\end{gather*}

For $a,b\in\mathbb{C}$, denote by $M_{a,b}$ the $\mathtt{C}$-module given as the quotient of 
the indecomposable projective module ${}_\mathtt{C}\mathtt{C}$ by the submodule $K_{a,b}$
generated by $ax+by$. 

\begin{lemma}\label{lem61}
Under the assumptions 
\begin{equation}\label{eq62}
a^2-ab+b^2=0,\qquad a\neq 0,\qquad b\neq 0, 
\end{equation}
the $\mathtt{C}$-module $M_{a,b}$ has the following properties:
\begin{enumerate}[$($i$)$]
\item\label{lem61.1} $M_{a,b}$ is a uniserial module of dimension $3$.
\item\label{lem61.2} $M_{a,b}$ is projective as a $\mathtt{C}^s$-module.
\item\label{lem61.3} $M_{a,b}$ is projective as a $\mathtt{C}^t$-module.
\item\label{lem61.4} The stable endomorphism algebra of $M_{a,b}$ (that is the quotient of 
$\mathrm{End}_{\mathtt{C}}(M_{a,b})$ modulo the ideal generated by all endomorphisms which factor
through a projective $\mathtt{C}$-module) is isomorphic to $\mathbb{C}$.
\end{enumerate}
\end{lemma}

\begin{proof}
The algebra $\mathtt{C}$ is self-injective and naturally graded (with both $x$ and $y$ having degree two).
As  $K_{a,b}$ is generated by a homogeneous element, both $K_{a,b}$ and $M_{a,b}$ are graded. 
The Poincar{\'e} polynomial 
of $\mathtt{C}$ is $1+2q^2+2q^4+q^6$. As $K_{a,b}$ is generated in degree two, its Poincar{\'e} polynomial
is either $q^2+2q^4+q^6$ or $q^2+q^4+q^6$. To determine which of the latter polynomials is the correct one,
we use the above multiplication table to 
write the result of the action of both $x$ and $y$ on the generator $ax+by$ of $K_{a,b}$ with respect to 
the basis $x^2$, $xy$ of the degree four part in $\mathtt{C}$. This gives the matrix
\begin{equation}\label{eq63}
\left(\begin{array}{cc}a&-b\\b&a-b\end{array}\right).  
\end{equation}
Under the assumption $a^2-ab+b^2=0$, the determinant of this matrix is zero. Therefore 
the Poincar{\'e} polynomial of $K_{a,b}$ is $q^2+q^4+q^6$. This means that the 
Poincar{\'e} polynomial of $M_{a,b}$ is $1+q^2+q^4$. Therefore $M_{a,b}$ has dimension $3$.
As $M_{a,b}$ also has simple top (since ${}_\mathtt{C}\mathtt{C}$ has simple top) 
and $\mathtt{C}$ is generated by elements of degree two, it follows that 
$M_{a,b}$ is uniserial, proving claim~\eqref{lem61.1}.
 
Claim~\eqref{lem61.2} follows from claim~\eqref{lem61.1} and observation that, in the case $a\neq 0$,
the element $xy$ which generates the socle of $\mathtt{C}^s$ does not belong to $K_{a,b}$ as the
degree four space in the latter is generated by $ax^2+bxy$. Similarly, claim~\eqref{lem61.3} 
follows from claim~\eqref{lem61.1} and observation that, in the case $b\neq 0$,
the element $x^2$ which generates the socle of $\mathtt{C}^t$ does not belong to $K_{a,b}$ as the
degree four space in the latter is generated by $ax^2+bxy$.
 
Consider the basis of $M_{a,b}$ given by the images of $1$, $x$ and $x^2$. Since $\mathtt{C}$ is commutative,
mapping $1$ to $x$ extends uniquely to an endomorphism $\varphi$ of $M_{a,b}$ and it follows that
$\mathrm{End}_{\mathtt{C}}(M_{a,b})\cong\mathbb{C}[\varphi]/(\varphi^3)$. To prove claim~\eqref{lem61.4},
we need to check that $\varphi$ factors through ${}_\mathtt{C}\mathtt{C}$. We claim that there is a
non-zero element $v$ of degree two in $\mathtt{C}$ such that $(ax+by)v=0$. Indeed, writing $v=\alpha x+\beta y$
and using our multiplication table, we get a system of linear equations for $\alpha$ and $\beta$ whose
matrix is exactly \eqref{eq63}. Under the assumptions of \eqref{eq62}, we have a non-zero solution to this system
of linear equations, for example, we can take $v=bx+ay$. 

For any $v$ as above, mapping $1$ to $v$ extends to a homomorphism 
$\psi:M_{a,b}\to {}_\mathtt{C}\mathtt{C}$.
Note that $a=\pm b$ together with $a^2-ab+b^2=0$ implies $a=b=0$. Therefore, under the 
assumptions of \eqref{eq62}, the elements $v$ and $ax+by$ are linearly independent in $\mathtt{C}$.
Hence, the composition of $\psi$ followed by the canonical projection of ${}_\mathtt{C}\mathtt{C}$
onto $M_{a,b}$ is non-zero and thus a non-zero scalar multiple of $\varphi$. This means that
$\varphi$ factors through ${}_\mathtt{C}\mathtt{C}$ and implies claim~\eqref{lem61.4}.
\end{proof}

For $a,b\in\mathbb{C}$ satisfying \eqref{eq62}, consider the subcategory $\mathcal{C}_{a,b}$ of 
$\mathcal{C}$ defined as the additive closure of the category $\mathcal{C}_{\mathrm{proj}}$ of
all projective objects in $\mathcal{C}$ together with $M_{a,b}$. Note that the 
action of $\cS$ restricts to $\mathcal{C}_{\mathrm{proj}}$ and the corresponding $2$-representation
$\mathbf{K}$ is equivalent to $\mathbf{C}_{\mathcal{L}_{sts}}$.

\begin{theorem}\label{thm64}
{\hspace{2mm}}

\begin{enumerate}[$($i$)$]
\item\label{thm64.1} The action of $\cS$ restricts to $\mathcal{C}_{a,b}$ and we denote by 
$\mathbf{M}_{a,b}$ the corresponding $2$-representation of $\cS$.
\item\label{thm64.2} We have a short exact sequence of $2$-representations as follows:
\begin{displaymath}
0\to  \mathbf{K}\to \mathbf{M}_{a,b}\to \mathbf{N}_{a,b}\to 0,
\end{displaymath}
where $\mathbf{N}_{a,b}$ is equivalent to $\mathbf{C}_{\mathcal{L}_{e}}$.
\item\label{thm64.3} The diagram of $\mathbf{M}_{a,b}$ is given 
(up to equivalence of $\mathbf{K}$ with $\mathbf{C}_{\mathcal{L}_{sts}}$ and of 
$\mathbf{N}_{a,b}$ with $\mathbf{C}_{\mathcal{L}_{e}}$) by
\begin{displaymath}
\xymatrix{ 
\mathbf{C}_{\mathcal{L}_{e}}\ar@{-}[d]|-{\{\theta_s,\theta_t,\theta_{st},\theta_{ts},\theta_{sts}\}}\\
\mathbf{C}_{\mathcal{L}_{sts}}
}
\end{displaymath}
\item\label{thm64.4} For two pairs $(a,b)$ and $(a',b')$ in $\mathbb{C}^2$ satisfying \eqref{eq6},
the following statements are equivalent.
\begin{enumerate}[$($a$)$]
\item\label{thm64.4.1}
The $2$-representations $\mathbf{M}_{a,b}$ and $\mathbf{M}_{a',b'}$ are equivalent.
\item\label{thm64.4.2}
The $\mathtt{C}$-modules $M_{a,b}$ and $M_{a',b'}$ are isomorphic.
\item\label{thm64.4.3}
We have $(a,b)=\lambda(a',b')$ for some non-zero $\lambda\in\mathbb{C}$.
\end{enumerate}
\end{enumerate}
\end{theorem}

\begin{proof}
On the level of $1$-morphism, $\cS$ is generated by $\theta_s$ and $\theta_t$. Therefore, to prove 
claim~\eqref{thm64.1}, it is enough to check that $\mathcal{C}_{a,b}$ is stable under the action of 
both $\theta_s$ and $\theta_t$. In terms of Soergel bimodules, the action of $\theta_s$ is given
by tensoring with $\mathtt{C}\otimes_{\mathtt{C}^s}\mathtt{C}$. By Lemma~\ref{lem61}\eqref{lem61.2},
the module $M_{a,b}$ is projective as $\mathtt{C}^s$-module. Therefore 
$\mathtt{C}\otimes_{\mathtt{C}^s}\mathtt{C}\otimes_{\mathtt{C}}M_{a,b}$ is projective as 
$\mathtt{C}$-module. Consequently, the action of $\theta_s$ preserves $\mathcal{C}_{a,b}$.
Similarly one shows that the action of $\theta_t$ preserves $\mathcal{C}_{a,b}$.
Claim~\eqref{thm64.1} follows.

To prove claim~\eqref{thm64.2}, we need to show that the quotient $\mathbf{N}_{a,b}$ defined
by the short exact sequence is equivalent to $\mathbf{C}_{\mathcal{L}_{e}}$.
From the previous paragraph we see that both $\theta_s$ and $\theta_t$ map $M_{a,b}$ to a projective 
$\mathtt{C}$-module. By Lemma~\ref{lem61}\eqref{lem61.3}, the stable endomorphism algebra of 
$M_{a,b}$ is trivial. However, this algebra coincides, by construction, with the endomorphism algebra
of the image of $M_{a,b}$ in $\mathbf{N}_{a,b}(\mathtt{i})$. Put together, this implies
claim~\eqref{thm64.2}. Claim~\eqref{thm64.3} follows now from claim~\eqref{thm64.2} and the observations
in the previous paragraph.

It remains to prove claim~\eqref{thm64.4}. Both equivalence of 
\eqref{thm64.4.2} and \eqref{thm64.4.3} and the fact that  \eqref{thm64.4.2} implies
\eqref{thm64.4.1} follow directly from constructions. So, we only need to show that 
\eqref{thm64.4.1} implies \eqref{thm64.4.2}. Assume that $\mathbf{M}_{a,b}$ and $\mathbf{M}_{a',b'}$ 
are equivalent and let $\Phi:\mathbf{M}_{a,b}(\mathtt{i})\to \mathbf{M}_{a',b'}(\mathtt{i})$ be 
the corresponding equivalence. Then we certainly have $\Phi(M_{a,b})\cong M_{a',b'}$. The endomorphism
algebra of the identity $1$-morphism in $\cS$ is, by construction, $\mathtt{C}$. This surjects,
via evaluation, onto the endomorphism algebra of both $M_{a,b}$ (with kernel $K_{a,b}$) and
$M_{a',b'}$ (with kernel $K_{a',b'}$). As $\Phi$ intertwines the actions of  $\cS$ on 
$\mathbf{M}_{a,b}(\mathtt{i})$ and $\mathbf{M}_{a',b'}(\mathtt{i})$, it follows that 
$K_{a,b}=K_{a',b'}$ and hence  $M_{a,b}={M}_{a',b'}$. This completes the proof.
\end{proof}

Theorem~\ref{thm64}\eqref{thm64.4} implies that there are, essentially, two different choices for
$M_{a,b}$, namely, $(a,b)=(1,\frac{1\pm\mathbf{i}\sqrt{3}}{2})$, where $\mathbf{i}^2=-1$.

\subsection{Analogous  discrete extension for $\cC_A\boxtimes \cC_{A^{\mathrm{op}}}$}\label{s6.3}

The example provided in the previous subsection motivates a series of examples described in the
present subsection.

Let $A$ be a basic and connected finite dimensional $\Bbbk$-algebra which we also assume to be non-simple.
Set $B:=A^{\mathrm{op}}$ and consider the enveloping algebra  $E:=A\otimes_{\Bbbk}B$. Let
$\mathcal{C}$ be a small category equivalent to
$E\text{-}\mathrm{mod}$. An endofunctor of $\mathcal{C}$ is called {\em good} provided that it belongs to 
the additive closure of the following functors:
\begin{itemize}
\item the identity functor;
\item tensoring with projective $E$-$E$--bimodules;
\item tensoring with the {\em left component projective} $E$-$E$--bimodule 
\begin{displaymath}
({}_AA\otimes_{\Bbbk}A_A)
\otimes_{\Bbbk}({}_{B}B_{B}); 
\end{displaymath}
\item tensoring with the {\em right component  projective}  $E$-$E$--bimodule 
\begin{displaymath}
({}_AA_A)\otimes_{\Bbbk}({}_{B}B\otimes_{\Bbbk}B_{B}). 
\end{displaymath}
\end{itemize}
If we view $E$-modules as $A$-$A$--bimodules (here the right action of $A$ corresponds to the left
action of $B$), then tensoring with left component projective $E$-$E$--bimodules
corresponds to tensoring with projective $A$-$A$--bimodules from the left. Similarly, tensoring with 
right component projective $E$-$E$--bimodules corresponds to tensoring with projective $A$-$A$--bimodules 
from the right. Denote by $\cD_A$ the $2$-category defined to have
\begin{itemize}
\item one object $\mathtt{i}$ which can be identified with $\mathcal{C}$;
\item as $1$-morphisms, all {\em good} endofunctors of $\mathcal{C}$;
\item as $2$-morphisms, all natural transformations of functors.
\end{itemize}
This $2$-category can be viewed as the $2$-category $\cC_A\boxtimes \cC_B$ in the notation of 
\cite[Subsection~6.1]{MM6}. The $2$-category $\cD_A$ has four two-sided cells:
\begin{itemize}
\item the two-sided cell $\mathcal{J}_1$ containing the identity $1$-morphism;
\item the two-sided cell $\mathcal{J}_2$ containing indecomposable functors corresponding to
left component projective bimodules;
\item the two-sided cell $\mathcal{J}_3$ containing indecomposable functors corresponding to
right component projective bimodules;
\item the two-sided cell $\mathcal{J}_4$ containing indecomposable functors corresponding to
projective bimodules. 
\end{itemize}
The Hasse diagram of the two-sided order is then given by
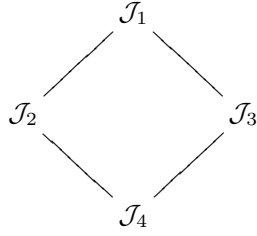
\begin{displaymath}
\xymatrix{
&\mathcal{J}_1\ar@{-}[dr]\ar@{-}[dl]&\\
\mathcal{J}_2\ar@{-}[dr]&&\mathcal{J}_3\ar@{-}[dl]\\
&\mathcal{J}_4&\\
}
\end{displaymath}
The two-sided cell $\mathcal{J}_1$ is also a left cell, which we denote by $\mathcal{L}_1$.
We fix some left cell $\mathcal{L}$ in $\mathcal{J}_4$.

Let $M$ be a $E$-module with the following properties:
\begin{itemize}
\item $M$ is indecomposable;
\item the stable endomorphism algebra of $M$ is trivial (i.e. isomorphic to $\Bbbk$);
\item $M$ is projective as an $A$-module;
\item $M$ is projective as a $B$-module.
\end{itemize}
We will comment on existence of such $M$ later ({\em a priori} there is no reason why such $M$ should exist), 
for the moment we just assume that it exists.  Let $\mathcal{C}_{\mathrm{proj}}$ 
be the category of projective objects
in $\mathcal{C}$ and $\mathcal{C}_{M}$ be the additive closure of $\mathcal{C}_{\mathrm{proj}}$ together
with $M$. Note that the action of $\cD_A$ restricts to $\mathcal{C}_{\mathrm{proj}}$, so we can denote by 
$\mathbf{K}$ the corresponding $2$-representation of $\cD_A$.
Then we have the following analogue of Theorem~\ref{thm64} in our present setup:

\begin{proposition}\label{prop71}
{\hspace{2mm}}

\begin{enumerate}[$($i$)$]\label{thm71}
\item\label{thm71.1} The $2$-representation $\mathbf{K}$ of $\cD_A$ is equivalent to $\mathbf{C}_\mathcal{L}$.
\item\label{thm71.2} The action of $\cD_A$ restricts to $\mathcal{C}_{M}$ and we denote by 
$\mathbf{M}_{M}$ the corresponding $2$-representation of $\cD_A$.
\item\label{thm71.3} We have a short exact sequence of $2$-representations as follows:
\begin{displaymath}
0\to  \mathbf{K}\to \mathbf{M}_{M}\to \mathbf{N}_{M}\to 0,
\end{displaymath}
where $\mathbf{N}_{M}$ is equivalent to $\mathbf{C}_{\mathcal{L}_{1}}$.
\item\label{thm71.4} The diagram of $\mathbf{M}_{M}$ is given by
\begin{displaymath}
\xymatrix{ 
\mathbf{C}_{\mathcal{L}_{i}}\ar@{-}[d]|-{X}\\
\mathbf{C}_{\mathcal{L}}
}
\end{displaymath}
where $X$ is a union of left cells which contains at least one left cell from $\mathcal{J}_2$,
$\mathcal{J}_3$ and $\mathcal{J}_4$.
\end{enumerate}
\end{proposition}

\begin{proof}
Claim~\eqref{thm71.1} is proved by the following standard argument, cf. \cite[Subsection~6.5]{MM2}
or \cite[Section~5]{MM5}.
Let $\mathrm{G}$ be the indecomposable $1$-morphism in $\mathcal{L}$ which corresponds to
tensoring with the $E$-$E$--bimodule $Ee\otimes_{\Bbbk}eE$, where $e\in E$ is a primitive idempotent. 
Let $L$ be a simple $E$-module such that $eL\neq 0$. Then restriction of the action of $\cD_A$ 
to the additive closure of the objects of the form $\mathrm{F}\, L$, where $\mathrm{F}\in\mathcal{L}$, 
defines an equivalence between $\mathbf{C}_\mathcal{L}$ and $\mathbf{K}$. This proves 
claim~\eqref{thm71.1}.

To prove claim~\eqref{thm71.2}, we just need to check that, for any $1$-morphism 
$\mathrm{F}\in\mathcal{J}_2\cup\mathcal{J}_3$, the $E$-module $\mathrm{F}\, M$ is projective.
Consider $M$ as an $A$-$A$--bimodule. Then we have
\begin{displaymath}
A\otimes_{\Bbbk}A\otimes_A M\cong A\otimes_{\Bbbk} M 
\end{displaymath}
which is projective as an $A$-$A$--bimodule since $M$ is projective as a right $A$-module.
This checks the necessary claim for $\mathrm{F}\in\mathcal{J}_2$. For $\mathrm{F}\in\mathcal{J}_3$,
the arguments are similar using right tensoring with projective $A$-$A$--bimodules and left $A$-projectivity of $M$.
This proves  claim~\eqref{thm71.2}. Claim~\eqref{thm71.3} follows immediately using the fact that 
$M$   has trivial stable endomorphism algebra.

Primitive idempotents in $E$ are of the form $e\otimes e'$, where  $e$ and $e'$ are primitive idempotents
in $A$. Let $L_{e,e'}$ be the corresponding simple $E$-module. Consider again $M$ as an $A$-$A$--bimodule.
Then a $1$-morphism $\mathrm{F}$ in $\mathcal{J}_2$ corresponds to left tensoring with some
$As\otimes_{\Bbbk}tA$, where $s$ and $t$ are primitive idempotents in $A$. The left cell containing
$\mathrm{F}$ is obtained by fixing $t$ and varying $s$. We have $\mathrm{F}\, L_{e,e'}\neq 0$ if and only if
$t=e$. This implies that the intersection $X\cap\mathcal{J}_2$ is a non-empty union of left cells.
A similar argument also works for $X\cap\mathcal{J}_3$ and $X\cap\mathcal{J}_4$. 
This proves  claim~\eqref{thm71.4} and completes the proof of our proposition.
\end{proof}

We note that the proof of Theorem~\ref{thm71}\eqref{thm71.4} also provides an explicit description of the set $X$
in terms of composition factors of $M$, cf. Theorem~\ref{thm32}\eqref{thm32.4}.

As mentioned above, existence of $M$ is not obvious in the general situation. The main problem is the
requirement that the stable endomorphism algebra of $M$ is trivial.
Below we would like to describe some situations in which a $E$-module $M$ satisfying the above conditions exists.

\begin{lemma}\label{lem71}
{\hspace{2mm}}

\begin{enumerate}[$($i$)$]
\item\label{lem71.1} If the center of $A$ is trivial, then one can take $M={}_AA_A$.
\item\label{lem71.2} If $\mathrm{char}(\Bbbk)\neq 2$ and $A=\Bbbk[x]/(x^2)$, then one can take $M={}_AA_A$.
\item\label{lem71.3} If $\mathrm{char}(\Bbbk)\neq 2$ and $A$ is the following quiver algebra:
\begin{displaymath}
\xymatrix{ 
1\ar@/^/[r]^{\alpha}&2\ar@/^/[r]^{\alpha}\ar@/^/[l]^{\beta}&
3\ar@/^/[r]^{\alpha}\ar@/^/[l]^{\beta}&\dots\ar@/^/[r]^{\alpha}\ar@/^/[l]^{\beta}&n\ar@/^/[l]^{\beta}
},\qquad\alpha^2=\beta^2=0,\quad \alpha\beta=\beta\alpha,
\end{displaymath}
then one can take $M={}_AA_A$.
\end{enumerate}
\end{lemma}

\begin{proof}
Clearly, ${}_AA_A$ is an indecomposable $A$-$A$--bimodule, which is projective both as a left and as a right
$A$-module. Further, $\mathrm{End}_{A\text{-}A}({}_AA_A)\cong Z(A)$, the center of $A$. This immediately implies
claim~\eqref{lem71.1}. 

To prove claim~\eqref{lem71.2}, we need to check that the stable endomorphism algebra of $A=\Bbbk[x]/(x^2)$ is
trivial. We have $\mathrm{End}_{A\text{-}A}({}_AA_A)=A$, so we just need to check that the endomorphism given
by multiplication with $x$ factors through projective $A$-$A$--bimodules.
There is a unique (up to scalar) homomorphism from $A$ to $A\otimes_{\Bbbk}A$ sending 
$1$ to $1\otimes x+x\otimes 1$. There is a unique (up to scalar) homomorphism from $A\otimes_{\Bbbk}A$ to $A$ sending 
$1\otimes 1$ to $1$. The composition $A\to A\otimes_{\Bbbk}A\to A$ thus sends $1$ to $2x$. 
As $\mathrm{char}(\Bbbk)\neq 2$, this composition is non-zero. This implies claim~\eqref{lem71.2}.

The quiver algebra in claim~\eqref{lem71.3} is naturally graded (all arrows have degree one). The
center of this algebra is also graded and has only elements of degree zero and two. The zero
part is one-dimensional, so to prove claim~\eqref{lem71.3} we just need to show that the degree
two part, which has dimension $n$, 
factors through projective $A$-$A$--bimodules. Here are the diagrams of projective  $A$-modules
(we simply write $i$ for the simple $A$-module corresponding to the vertex $i$):
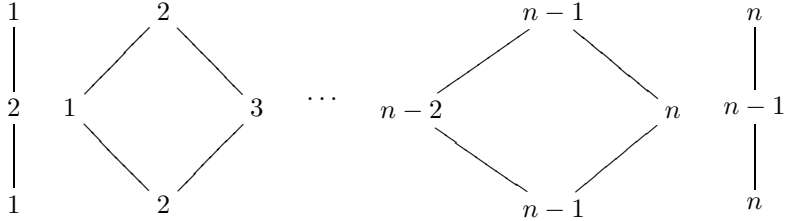
\begin{displaymath}
\xymatrix{1\ar@{-}[d]\\2\ar@{-}[d]\\1}\quad
\xymatrix{&2\ar@{-}[dr]\ar@{-}[dl]&\\1\ar@{-}[dr]&&3\ar@{-}[dl]\\&2&}\quad
\xymatrix{\\\dots\\}\quad
\xymatrix{&n-1\ar@{-}[dr]\ar@{-}[dl]&\\n-2\ar@{-}[dr]&&n\ar@{-}[dl]\\&n-1&}\quad
\xymatrix{n\ar@{-}[d]\\n-1\ar@{-}[d]\\n}
\end{displaymath}
Further, the diagram of the $A$-$A$--bimodule ${}_AA_A$ is as follows
(we simply write $ij$ for the simple $A$-$A$--bimodule corresponding to the pair $(i,j)$ of vertices):
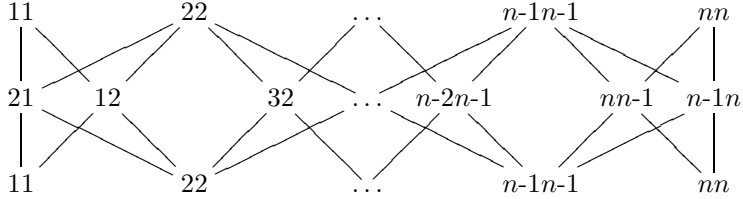
\begin{displaymath}
\xymatrix@!=0.7pc{
11\ar@{-}[d]\ar@{-}[dr]&  &22\ar@{-}[dr]\ar@{-}[drr]\ar@{-}[dl]\ar@{-}[dll]&  
&\dots\ar@{-}[dl]\ar@{-}[dr]&      &\text{$n$-$1$$n$-$1$}\ar@{-}[dr]\ar@{-}[drr]\ar@{-}[dl]\ar@{-}[dll]&    
&nn\ar@{-}[d]\ar@{-}[dl]\\
21\ar@{-}[d]\ar@{-}[drr] &12\ar@{-}[dr]\ar@{-}[dl]&  
&32\ar@{-}[dr]\ar@{-}[dl]&\dots\ar@{-}[dll]\ar@{-}[drr]&
\text{$n$-$2$$n$-$1$}\ar@{-}[dr]\ar@{-}[dl]&      
&\text{$n$$n$-$1$}\ar@{-}[dr]\ar@{-}[dl]&\text{$n$-$1$$n$}\ar@{-}[d]\ar@{-}[dll]\\
11&  &22&  &\dots&      &\text{$n$-$1$$n$-$1$}&    &nn
}
\end{displaymath}
Consider the subquotient $Q$ of the latter given by 
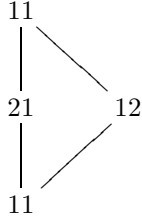
\begin{displaymath}
\xymatrix{
11\ar@{-}[d]\ar@{-}[dr]&\\
21\ar@{-}[d] &\ar@{-}[dl]12\\
11&  
}
\end{displaymath}
The diagram of the projective $A$-$A$--bimodule $P_{11}$ with top $11$ is as follows:
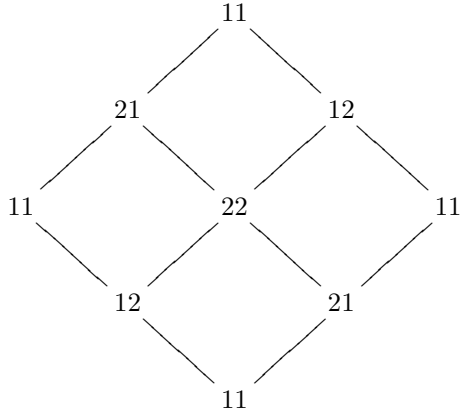
\begin{displaymath}
\xymatrix{
&&11\ar@{-}[dl]\ar@{-}[dr]&&\\
&21\ar@{-}[dl]\ar@{-}[dr]&&12\ar@{-}[dl]\ar@{-}[dr]&\\
11\ar@{-}[dr]&&22\ar@{-}[dl]\ar@{-}[dr]&&11\ar@{-}[dl]\\
&12\ar@{-}[dr]&&21\ar@{-}[dl]&\\
&&11&&\\
}
\end{displaymath}
There is a unique, up to scalar, homogeneous embedding from $Q$ to an appropriate graded shift of $P_{11}$ 
and  also unique, up to scalar, homogeneous projection from $P_{11}$ onto $Q$. As $\mathrm{char}(\Bbbk)\neq 2$, 
the corresponding composition  $Q\hookrightarrow P_{11}\tto Q$ is non-zero. 
Indeed, the embedding $Q\hookrightarrow P_{11}$ maps the generator of $Q$ to the sum of the two
$11$-components which are in the middle layer of $P_{11}$, while the kernel of the projection $P_{11}\tto Q$
contains the difference of these components. This implies 
that the nilpotent endomorphism of the first projective $A$-module is annihilated in the stable category. 
Similarly one shows that all nilpotent endomorphisms of all projective $A$-module are annihilated in the 
stable category. This implies claim~\eqref{lem71.3} and completes the proof.
\end{proof}

Both the statement of Lemma~\ref{lem71}\eqref{lem71.3} and its proof generalize to arbitrary finite-dimensional 
zigzag algebras from \cite{HK}.


\noindent
Aaron Chan: Department of Mathematics, Uppsala University,
Box 480, SE-751~06, Uppsala, SWEDEN, {\tt aaron.chan\symbol{64}math.uu.se}

\noindent
Volodymyr Mazorchuk: Department of Mathematics, Uppsala University,
Box 480, SE-751 06, Uppsala, SWEDEN, {\tt mazor\symbol{64}math.uu.se}

\end{document}